 \documentclass[educ, sglanonrev]{arxiv_version}
\usepackage{tgtermes}
\usepackage{newtxtext}
\usepackage{newtxmath}
\usepackage{bm}
\usepackage{endnotes}
\OneAndAHalfSpacedXII % Current default line spacing
\usepackage{algpseudocode}
\usepackage{tikz}
\usetikzlibrary{arrows.meta}
\usepackage{fix-cm}
\usepackage{subcaption}
\usepackage{comment}
\newif\ifdiscrete
\discretefalse   % \discretetrue or \discretefalse
\ifdiscrete
  \excludecomment{CTonly}
  \newenvironment{DTonly}{}{}
\else
  \excludecomment{DTonly}
  
\fi
\usepackage{hyperref}
\usepackage{cleveref}
\newcommand{\mathbbm}[1]{\text{\usefont{U}{bbm}{m}{n}#1}}
\newcolumntype{P}[1]{>{\centering\arraybackslash}p{#1}}
\newcolumntype{N}{@{}m{0pt}@{}}
\newcommand{\E}[1]{\mathbb{E}\!\left[#1\right]}
\newcommand{\bE}{\mathbb{E}}
\newcommand{\new}[1]{#1}
\usepackage[sort&compress]{natbib}
 \bibpunct[, ]{[}{]}{,}{n}{}{,}%
 \def\bibfont{\small}%
\EquationsNumberedThrough    % Default: (1), (2), ...
\TheoremsNumberedThrough     % Preferred (Theorem 1, Lemma 1, Theorem 2)
\ECRepeatTheorems  %
\MANUSCRIPTNO{EDUC-0001-2026.00}
\begin{document}
%%%%%%%%%%%%%%%%
\RUNAUTHOR{}
\RUNTITLE{Transform Method for Stochastic Processing and Matching Networks}
\TITLE{Transform Method for Stochastic Processing and Matching Networks}
\ARTICLEAUTHORS{%
\AUTHOR{Sushil Mahavir Varma}
\AFF{Industrial and Operations Engineering, University of Michigan Ann Arbor \EMAIL{sushilv@umich.edu}}
\AUTHOR{Prakirt Raj Jhunjhunwala}
\AFF{Amazon, \EMAIL{prakirt2203@gmail.com}}
\AUTHOR{Daniela Hurtado-Lange}
\AFF{Kellogg School of Management, Northwestern University \EMAIL{daniela.hurtado@kellogg.northwestern.edu}}
\AUTHOR{Siva Theja Maguluri}
\AFF{Industrial and Systems Engineering, Georgia Institute of Technology \EMAIL{siva.theja@gatech.edu}}
} % end of the block
\ABSTRACT{%
\new{Modern service systems---ranging from cloud data centers and ride-hailing platforms to healthcare facilities---operate at massive scales where it is important to handle congestion. Queueing theory is used to understand the delay and queue length behavior in these systems. Except in simple queues, it is not possible to obtain a closed form solution for the quantities of interest, and so, one studies the system in certain asymptotic regimes such as the heavy traffic. The transform method, presented in this tutorial, is a framework to understand the steady-state behavior of Stochastic Processing and Matching Networks (SPNs/SMNs). By exploiting the zero-drift property of exponential test functions, the method derives explicit \textit{functional equations} (acting as a proxy for global balance equations) for the transforms (such as moment-generating functions) of queue-length distributions. These functional equations can be used to either characterize the exact behavior of the system in an asymptotic regime or to obtain non-asymptotic performance bounds on the mean, higher order moments, or tail bounds on the queue lengths.
Since its introduction for load-balancing in data center networks, the transform method, as a framework, has been extended to handle various features that arise in different systems, including customer abandonment, state-dependent arrivals, Markov-modulated arrivals, large-system scale, and multi-dimensional networks with multiple bottlenecks. This tutorial presents an overview of the transform method starting with the simplest setting viz., a single server queue. The transform method is introduced as a three step procedure. We then illustrate how the method can be adapted within this three-step framework to handle the aforementioned features.}
}%
\FUNDING{This work was partially supported by NSF grant EPCN-2144316 and CMMI-2140534.}
\KEYWORDS{Transform Method, Stochastic Networks, Steady-State Analysis, Drift, Service Systems}
\maketitle

\section{Introduction}
In our increasingly interconnected world, the efficiency of service systems\new{,} ranging from emergency departments and ride sharing systems to AI data centers and quantum communication networks\new{,} has become \new{important}. The operation of these systems requires dealing with inherent uncertainty arising due to the stochastic nature of events governing the dynamics of the process.
\new{Queueing theory is typically used to} tame this uncertainty by developing stylized models that yield both interpretable and implementable decision rules.
In this tutorial, we \new{consider} a class of complex service systems, formally known as Stochastic Processing (\citet{williams2016stochastic}) and Matching (\citet{mahavir2025stochastic}) Networks (SPNs/SMNs). In these models, \new{jobs arriving at random times needing a random amount of service are governed by} scheduling, routing, or matching policies, giving rise to stochastic processes whose behavior \new{determines the} system performance. The central analytical challenge is to characterize the steady-state queue-length distribution so that one can \new{use these} performance guarantees \new{to} make informed design decisions. \new{However, except in simple cases, SPNs/SMNs result in complex processes which are intractable, and closed form exact expressions are impossible to obtain. Therefore, one adopts an asymptotic viewpoint which enables one to characterize the exact behavior. A common regime that is studied is the heavy-traffic regime.}

In recent years, the transform method has emerged as a unified and \new{tractable} framework for the steady-state analysis of SPNs and SMNs (\citet{mahavir2025stochastic, jhunjhunwala2024design, lange2021asymptotic}). The core idea is to apply an \emph{exponential} test function to the queue dynamics and exploit the zero-drift condition in steady state. Because the exponential generating function encodes all moments simultaneously, a single calculation captures the entire queue-length distribution, yielding an exact \emph{functional equation} for the generating function of the steady-state queue length. This direct \new{approach not only provides exact characterization in the heavy-traffic asymptotic regime,} but also yields performance guarantees in non-asymptotic regimes.

The practical relevance of the transform method lies in its ability to deliver actionable performance guarantees in \new{stochastic networks}. Operators of these systems routinely need tight guarantees on \new{the mean delay, mean queue length, and} tail behavior, \new{for} service-level compliance. The transform method offers a direct way to obtain such guarantees because it produces explicit functional constraints on the transform of steady-state queue lengths, which can then be translated into quantitative bounds on \new{the mean, moments, or the tail.} This makes the method particularly useful in practice for designing routing, scheduling, or pricing controls: Complementing simulation \new{or} heuristics, one can employ the resulting transform-based inequalities to certify that a chosen control policy will meet performance thresholds at realistic scales. This tutorial demonstrates the simplicity, generalizability, and power of the transform method that adapts naturally to the operational features of modern OR applications.
These include the following.
\begin{itemize}
    \item \new{\textit{Customer abandonment.} In call centers, emergency departments, and ride-hailing systems, customers who wait too long abandon before being served. Abandonment fundamentally changes the system dynamics making it more challenging to analyze. However, the transform method is flexible enough to analyze such a system and characterize the resulting phase transition in the queue-length distribution, informing staffing and capacity decisions (\citet{jhunjhunwala2026jsqa}).}

    \item \new{\textit{Markov-modulated environments.} In wireless networks and data centers, channel conditions and demand patterns fluctuate according to an exogenous stochastic environment, introducing temporal correlations that break the independence assumptions of classical models. Such non-independent fluctuations can be modeled with Markov-modulated arrivals. The transform method, combined with the Poisson equation, can be used to absorb the correlations into an effective variance that governs the heavy-traffic distribution (\citet{HL-Gro-2026-Markov-Modulated}).}

    \item \textit{State-dependent arrivals.} In two-sided marketplaces such as ride-hailing platforms, the platform adjusts prices based on current congestion, making arrival rates depend on the system state. The transform method characterizes the distribution of the queue lengths in heavy traffic for a matching queue with dynamic/state-dependent arrivals (\citet{varma2022twosidedqueues}). This shows how pricing policies shape the steady-state congestion distribution, enabling platform designers to evaluatepricing strategies analytically.

    \item \textit{Large-scale load balancing.} \new{Today's} data centers route millions of requests across thousands of servers, and the routing policy must balance load to minimize delay. \new{Under popular load-balancing strategies, the system behaves as if it has a single bottleneck, and this phenomenon is called the state space collapse (SSC), and the system is said to satisfy the complete resource pooling condition. The transform method works in conjunction with the SSC and} yields explicit guarantees on the tail of \new{queue lengths}. The method also provides tight pre-limit error bounds in non-asymptotic scenarios, bridging the gap between asymptotic theory and the real-world performance (\citet{hurtado2020transform, jhunjhun2024exponential}).

    \item \textit{Multi-dimensional networks with multiple bottlenecks.} \new{An input-queued switch enables connections between several computers, and it enables connections between its input ports and output ports. Such switches not only power the internet, but also serve as an approximation to model data center networks. A switch is a representative of a multiple bottleneck system and exhibits SSC into a multi-dimensional subspace. The transform method can be used to characterize the \textit{joint} distribution of queue lengths in heavy traffic. This is a first step towards a broader understanding of systems that do not satisfy the complete resource pooling condition} (\citet{Jhun_heavy_traffic}).
\end{itemize}
\new{The transform method is a three-step procedure consisting of deriving the transform equation, approximating to second order and solve the functional equation. Each of the above mentioned extensions uses the same template while introducing new techniques to handle the challenges due to the richer structure.} Taken together, these applications illustrate that the transform method offers a tractable, adaptable, and broadly applicable framework for analyzing, designing, and controlling real-world stochastic \new{networks}.

\textit{How to read this tutorial.}
We assume the reader has graduate-level background in probability and
discrete-time Markov chains, along with basic familiarity with transforms
(moment generating and characteristic functions). Prior exposure to
heavy-traffic limits for the single-server queue not
required since we develop the necessary background as we go. The remainder of
the tutorial is organized in three tiers. Section~\ref{sec:transform} is the
\emph{introductory core}: it develops the three-step transform-method framework
for the discrete-time single-server ($G/G/1$) queue and derives the heavy-traffic
limit; it is largely self-contained and
suffices for a first reading. Section~\ref{sec:single-server-variants} is
\emph{intermediate}: it stays in the single-queue setting but illustrates
how the same three-step framework combines with customer abandonment (Section~\ref{sec:abandonment}), Markov-modulated arrivals (Section~\ref{sec:markov-modulated}), and state-dependent arrivals (Section~\ref{sec:matching-queue}), requiring basic knowledge of differential equations, the Poisson equation, and inverse
Fourier transforms, respectively. Section~\ref{sec:networks} is the most \emph{advanced}: it
scales to multi-dimensional stochastic networks and assumes some familiarity
with existing results in heavy-traffic theory.
Readers primarily interested in the method itself can stop after
Section~\ref{sec:single-server-variants}; those interested in network-level
applications should continue through Section~\ref{sec:networks}.

\section{The Transform Method: A Single-Server Queue Illustration}\label{sec:transform}
In this section, we introduce the transform method by developing it in the simplest possible setting: the discrete-time single-server queue.  We begin by defining the model and recalling the classical Kingman-type moment bounds (\citet{kingman}) obtained via quadratic Lyapunov functions.  We then explain why these polynomial-based approaches become \new{cumbersome} for characterizing the full distributional picture, which motivates the use of \emph{exponential} test functions.  After walking through the key steps of the method, we discuss the choice of transform,
explain how tail bounds follow naturally, and close with a comparison to other methods for steady-state analysis of queueing systems.

\subsection{The Discrete-Time Single-Server Queue}\label{sec:model}
Consider a discrete-time $G/G/1$ queue defined as follows. At each time slot $k = 0, 1, 2, \ldots$, a random number $a(k)$ of customers arrive to the queue, and the server can potentially serve $s(k)$ customers.  We assume $\{a(k)\}$ and $\{s(k)\}$ are i.i.d.\ sequences (independent of each other) with
\begin{equation*}
\mathbb{E}[a(k)] = \lambda, \quad \mathrm{Var}(a(k)) = \sigma_a^2, \qquad \mathbb{E}[s(k)] = \mu, \quad \mathrm{Var}(s(k)) = \sigma_s^2,
\end{equation*}
and that the system is stable, i.e., $\lambda < \mu$.  The \emph{heavy-traffic parameter} is
\[
\epsilon \;:=\; \mu - \lambda,
\]
which measures the system's slack: it is the
per-slot excess service capacity over the arrival rate. Thus, small
$\epsilon$ means a heavily loaded queue, and the regime $\epsilon \to 0$
captures the setting where delays are largest. Throughout the rest of the tutorial, ``heavy traffic''
will always mean $\epsilon \to 0$. For ease of analysis, also assume the existence of $A_{\max}, S_{\max} > 0$ such that $a(1) \leq A_{\max}$ and $s(1) \leq S_{\max}$ with probability 1. The queue length evolves according to the \emph{Lindley-type recursion}:
\begin{equation}\label{eq:lindley}
q(k+1) \;=\; \bigl[q(k) + a(k) - s(k)\bigr]^{+},
\end{equation}
where $[\,\cdot\,]^+ = \max(\,\cdot\,, 0)$. The queue length at the next slot equals the current queue length plus the arrivals that came in this slot minus the services completed this slot, with the result floored at zero (since the queue cannot go negative). Equation \eqref{eq:lindley} can be rewritten by introducing an additional variable as
\begin{equation}\label{eq:queue-evolution}
q(k+1) \;=\; q(k) + a(k) - s(k) + u(k),
\end{equation}
where $u(k)$ is the \emph{unused service} at time $k$.
The unused service represents the service that is wasted because the queue is empty: it is the difference between the \textit{potential} service
and the \emph{actual} service that was used.  By definition, $u(k) \ge 0$, and the following complementarity condition holds:
\begin{equation}\label{eq:complementarity}
q(k+1) \cdot u(k) = 0,
\end{equation}
as unused service is positive only when the queue drains to zero.  The
queue evolution~\eqref{eq:queue-evolution} and the complementarity condition~\eqref{eq:complementarity} completely characterize the queue evolution.

\subsection{Kingman's Bound via the Quadratic Lyapunov Function}\label{sec:kingman}
A classical approach to heavy-traffic analysis of the single-server queue uses a \emph{quadratic} Lyapunov function (or test function).  The idea is simple: apply the function $\varphi(q) = q^2$ to the queue evolution and use the fact that, in steady state, the expected ``drift'' of this function must be zero. The key observation is that by using the complementarity condition~\eqref{eq:complementarity}, we get
\[
q(k+1)^2 \;=\; \bigl(q(k) + a(k) - s(k)\bigr)^2 - u(k)^2,
\]
where we use $2\bigl(q(k) + a(k) - s(k)\bigr)\cdot u(k) = 2(q(k+1) -u(k))\cdot u(k) = -2 u(k)^2$.
As we showcase later in Section~\ref{sec:single-server-variants} and Section~\ref{sec:networks}, this use of the complementarity condition~\eqref{eq:complementarity} can be generalized to a large extent and is one of the core reasons behind the technical simplicity of the transform method. Next, setting the steady-state drift to zero, i.e., $\mathbb{E}[q(k+1)^2] = \mathbb{E}[q(k)^2]$, and expanding:
\begin{align}
0 &= \mathbb{E}\bigl[q(k+1)^2 - q(k)^2\bigr] = \mathbb{E}\bigl[(a-s)^2 + 2q(a-s) - u^2\bigr] = \sigma_a^2 + \sigma_s^2 + \epsilon^2 - 2\epsilon\,\mathbb{E}[q] - \mathbb{E}[u^2], \label{eq:quadratic_drift}
\end{align}
where we omitted the dependence on $k$ in steady state. In the last equality, we used the definition of variance to obtain $\mathbb{E}[(a-s)^2] = \sigma_a^2 + \sigma_s^2 + \epsilon^2$, and the definition of $\epsilon$ to obtain $\mathbb{E}[q(a-s)] = -\epsilon\,\mathbb{E}[q]$ by independence.  Rearranging, we get
$
2\epsilon\,\mathbb{E}[q] \;=\; \sigma_a^2 + \sigma_s^2 + \epsilon^2 - \mathbb{E}[u^2]
$.
To bound the unused service, we can simply take expectation in steady state in the queue evolution equation \eqref{eq:queue-evolution} and note that $\E{q(k+1)} = \E{q(k)}$ to get $\mathbb{E}[u] = \mathbb{E}[s-a] = \epsilon$. Thus, we have $0\leq \mathbb{E}[u^2] \leq S_{\max} \mathbb{E}[u] = \epsilon S_{\max}$, where we used that $0 \leq u \leq s \leq S_{\max}$ with probability 1. This bound on the unused service gives the celebrated \emph{Kingman bound}:
\begin{equation}\label{eq:kingman}
\frac{\sigma_a^2 + \sigma_s^2}{2\epsilon} - \frac{S_{\max}}{2} \leq \mathbb{E}[q] \leq \frac{\sigma_a^2 + \sigma_s^2}{2\epsilon} + \frac{\epsilon}{2}
\end{equation}
Note that the above equation holds for any $\epsilon > 0$, thus providing useful bounds on $\mathbb{E}[q]$ for any traffic. Moreover, rescaling by $\epsilon$ and taking $\epsilon \to 0$ gives $\mathbb{E}[\epsilon q] \to (\sigma_a^2 + \sigma_s^2)/2$, thus providing a tight expression for the expected scaled queue length in the heavy traffic limit.

In practice, one may need a handle on higher moments and tails beyond the mean to provide tight service level guarantees. To do this, one can go further and compute higher moments of $q$ by using higher-degree polynomial test functions: $\varphi(q) = q^3$ for the second moment, $\varphi(q) = q^4$ for the third, and so on.  In principle, the $n$-th moment can be bounded in terms of all lower moments using $\varphi(q) = q^{n+1}$ as the test function, so an induction argument applies. Using this method, \citet{atilla} obtain
\begin{align*}
    \bigg|\mathbb{E}[\epsilon^n q^n] - n! \left(\frac{\sigma_a^2 + \sigma_s^2}{2}\right)^n\bigg| = O(\epsilon) \quad \forall n \in \{1, 2, \hdots\},
\end{align*}
which matches the moments of an $\mathrm{Expo}\bigl(2/(\sigma_a^2 + \sigma_s^2)\bigr)$\new{, i.e., an exponential random variable with mean $(\sigma_a^2+\sigma_s^2)/2$. This suggests} that the scaled queue length $\epsilon q$ converges to an exponential distribution. Note that matching all moments does not characterize a distribution in general---this is the classical \emph{moment problem}---but when the moments do not grow too fast (Carleman's condition,
which the limiting exponential here satisfies; e.g., see \cite{koehler2025constructive}), they
uniquely determine the distribution. As one can imagine, this \new{inductive} approach becomes increasingly \new{tedious} as the underlying model becomes richer.  Each step requires careful bookkeeping of cross-terms and the use of previously established bounds to control the lower-order terms.

A simple and quite powerful idea here is to use an \emph{exponential}
test function in place of polynomial ones. Exponentials offer two
key advantages. First, they encode all moments simultaneously via the
series $e^{\theta\epsilon q} = \sum_{n=0}^{\infty}
(\theta\epsilon)^n q^n / n!$. Second, they are algebraically clean
under the Lindley-type update: whereas $(q + \Delta)^n$ requires a
binomial expansion (as
in~\eqref{eq:quadratic_drift}), the exponential factors as
$e^{\theta\epsilon(q+\Delta)} = e^{\theta\epsilon q} \cdot
e^{\theta\epsilon\Delta}$ (where $\Delta$ is any additive increment
such as $a - s$). Additive updates thus become multiplicative factors, \new{leading to simpler algebra}.

\subsection{The Transform Method}\label{sec:transform-method}
The central idea of the transform method, introduced by
\citet{hurtado2020transform}, is to use the exponential test function
$\varphi(q) = e^{\theta \epsilon q}$ for a carefully chosen parameter
$\theta$. Another way to see why this is powerful, recall the pattern of
Kingman's argument: loosely speaking, the zero-drift condition for a test function $V$
produces a bound on $\mathbb{E}[V'(q)]$, the expectation of its derivative in steady state for a single-server queue. For polynomial $V = q^{n+1}$, this yields the \new{$n$th moment. For} exponential
$V = e^{\theta\epsilon q}$\new{,} since the derivative is $\theta\epsilon\,V$\new{,} the same argument produces a bound on the
MGF $\mathbb{E}[e^{\theta\epsilon q}]$, which encodes all moments at
once. The method proceeds in three steps.

\subsubsection*{Step 1: Capturing the System Dynamics \new{and Deriving the Transform Equation}}
Note that the distribution of $q(k+1)$ is the same as that of $q(k)$ \new{in steady state}. For the time being, assuming that $\mathbb{E}[e^{\theta \epsilon q}] < \infty$ in the steady-state, we get
\begin{equation} \label{eq:zero_drift}
\mathbb{E}\bigl[e^{\theta\epsilon\, q(k+1)}\bigr] \;=\; \mathbb{E}\bigl[e^{\theta\epsilon\, q(k)}\bigr].
\end{equation}
The key lemma that connects the queue evolution to this exponential function exploits the complementarity condition~\eqref{eq:complementarity}. \new{From the definition of the unused service, we know that whenever the unused service kicks in, i.e., $u(k) > 0$, we have that $q(k+1) = 0$, which implies}
\begin{equation}\label{eq:key-lemma}
\bigl(e^{\theta\epsilon\, q(k+1)} - 1\bigr)\bigl(e^{-\theta\epsilon\, u(k)} - 1\bigr) = 0.
\end{equation}
To see why, note that $e^x - 1 = 0$ if and only if $x = 0$, and \new{so, when $u(k)=0$, the above product is zero, and when $u(k)>0$, we have $q(k+1)=0$, again making the product zero.} Expanding~\eqref{eq:key-lemma} and \new{replacing $q(k+1)$} using the queue evolution~\eqref{eq:queue-evolution}, one obtains after taking steady-state expectations:
\begin{equation}\label{eq:transform-equation-before-independence}
    \mathbb{E}\bigl[e^{\theta\epsilon q}(1-e^{\theta \epsilon(a-s)}) \bigr] = 1 - \mathbb{E}\bigl[e^{-\theta \epsilon u} \bigr].
\end{equation}
where we omitted the dependence on $k$ in steady state. The only properties used so far are stationarity and the complementarity condition, which makes \eqref{eq:transform-equation-before-independence} especially general for a variety of single-server queue models as we will see in the next section. If one further uses that the number of arrivals and potential services are independent \new{of} the queue length, one can rearrange \eqref{eq:transform-equation-before-independence} to obtain
\begin{equation}\label{eq:transform-equation}
\mathbb{E}\bigl[e^{\theta\epsilon q}\bigr] \;=\; \frac{1 - \mathbb{E}\bigl[e^{-\theta\epsilon u}\bigr]}{1 - \mathbb{E}\bigl[e^{\theta\epsilon(a - s)}\bigr]}.
\end{equation}
This is the \emph{transform equation}.  It expresses the moment generating function (MGF) of the scaled queue length in terms of the MGF of the unused service and the net input $a - s$. Equation~\eqref{eq:transform-equation} is remarkable because it is \emph{exact}---it holds for any $\epsilon > 0$, not just in the heavy-traffic limit.  Essentially, the exponential test function \new{handles the nonlinearity of the Lindley recursion through the complementarity condition via \eqref{eq:key-lemma} in an elegant manner.}

\subsubsection*{Step 2: Second-Order Approximation}
We expand the exponential terms using Taylor expansion for small $\epsilon$. Since $\mathbb{E}[a - s] = -\epsilon$ and $\mathrm{Var}(a-s) = \sigma_a^2 + \sigma_s^2$, the denominator becomes
\[
1 - \mathbb{E}\bigl[e^{\theta\epsilon(a-s)}\bigr] = \theta\epsilon^2 - \frac{\theta^2\epsilon^2}{2}(\sigma_a^2 + \sigma_s^2) + o(\epsilon^2).
\]
For the numerator, we use $\mathbb{E}[u]=\epsilon$ and $\mathbb{E}[u^2] = O(\epsilon)$ as discussed in \Cref{sec:kingman}, to get
\begin{equation}
1-\mathbb{E}\bigl[e^{-\theta\epsilon u}\bigr] = \theta\epsilon\,\mathbb{E}[u] - \frac{\theta^2\epsilon^2}{2}\,\mathbb{E}[u^2] + O(\epsilon^3) = \theta \epsilon^2 + O(\epsilon^3) \label{eq:unused_service}
\end{equation}
Substituting these expansions into the transform
equation~\eqref{eq:transform-equation} and dividing the numerator and
denominator by their common leading factor $\theta\epsilon^2$, one arrives at
\begin{equation}\label{eq:approx-mgf}
\mathbb{E}\bigl[e^{\theta\epsilon q}\bigr] \;=\; \frac{1 + o(1)}{1 - \theta\left(\frac{\sigma_a^2 + \sigma_s^2}{2}\right) + o(1)}.
\end{equation}

\subsubsection*{Step 3: Solving the Functional Equation}
For the single-server queue, the final step is straightforward: take $\epsilon \to 0$ in~\eqref{eq:approx-mgf} to obtain
\begin{equation}\label{eq:limiting-mgf}
\lim_{\epsilon \to 0}\mathbb{E}\bigl[e^{\theta\epsilon q}\bigr] \;=\; \frac{1}{1 - \theta\left(\frac{\sigma_a^2 + \sigma_s^2}{2}\right)}.
\end{equation}
\new{By carefully keeping track of the conditions on $\theta$ under which $\mathbb{E}[e^{\theta \epsilon q}] < \infty$, one can show that this equality holds true for $\theta$ in an open interval around the origin.}
The right-hand side is the MGF of an exponential random variable with mean $(\sigma_a^2 + \sigma_s^2)/2$.  Since convergence of MGFs in a neighborhood of the origin implies convergence in distribution \cite[Theorem 9.5 in Section 5]{gut2012probability}, we conclude
\begin{equation}\label{eq:ht-result}
\epsilon q \;\xrightarrow{d}\; \mathrm{Expo}\!\left(\frac{2}{\sigma_a^2 + \sigma_s^2}\right) \qquad \text{as } \epsilon \to 0.
\end{equation}
This is the heavy-traffic result for the discrete-time $G/G/1$ queue.

\paragraph{Existence of the MGF.} Step~1 assumed
$\mathbb{E}[e^{\theta\epsilon q}] < \infty$ in steady state, but this
finiteness must be \emph{established} before the drift-zero condition \eqref{eq:zero_drift}
can be invoked. The standard route is via a \emph{negative drift}
bound: one shows that there exist constants $\delta > 0$ and
$C < \infty$ such that
\begin{equation}
\label{eq:mgf_finite_argument}
\mathbb{E}\!\left[e^{\theta\epsilon q(k+1)} \;\big|\; q(k) = q\right]
- e^{\theta\epsilon q}
\;\le\; -\delta(\theta)\,e^{\theta\epsilon q} + C
\qquad \text{for all } q \ge 0,
\end{equation}
and then appeals to a moment-bound lemma (see \citet[Proposition 6.14]{hajekrandomprocbook}) to conclude $\mathbb{E}[e^{\theta\epsilon q}]
\le C/\delta(\theta) < \infty$ for all $\theta$ such that $\delta(\theta)>0$. \new{Thus, t}he negative-drift
bound certifies that the steady-state MGF is
finite, and also gives the corresponding range of $\theta$\new{, which we can show includes an open interval around the origin}. The transform method requires both steps: first, establish finiteness via a negative drift, then derive the exact transform equation by setting the drift to zero. Throughout the rest of the tutorial, we focus on the latter while only making brief remarks about the former.

\subsection{Empirical Validation}
Figure~\ref{fig:ht-convergence} illustrates the heavy-traffic
convergence~\eqref{eq:ht-result} for two contrasting arrival and service
distributions. The code used to generate these plots is publicly available at \citet{varma2026htsim}. In the Bernoulli case the queue length is a birth--death
chain on the non-negative integers, \new{and} its stationary distribution is
geometric for every $\epsilon > 0$; the heavy-traffic limit then reduces
to a geometric random variable converging to an exponential as
$\epsilon \to 0$, and what one observes empirically is essentially the
discretization gap shrinking with $\epsilon$. The bursty case
underscores the \emph{universality} of the limit: although the arrival
distribution is multi-modal and far from Bernoulli, the empirical
density still \new{converges to} an exponential, and \new{the limit depends on} the arrival and service processes
\new{only through} their first two moments \new{via the expression} $\sigma_a^2 + \sigma_s^2$.
\begin{figure}[ht]
\centering
\begin{minipage}[t]{0.475\textwidth}
\centering
\includegraphics[width=\textwidth]{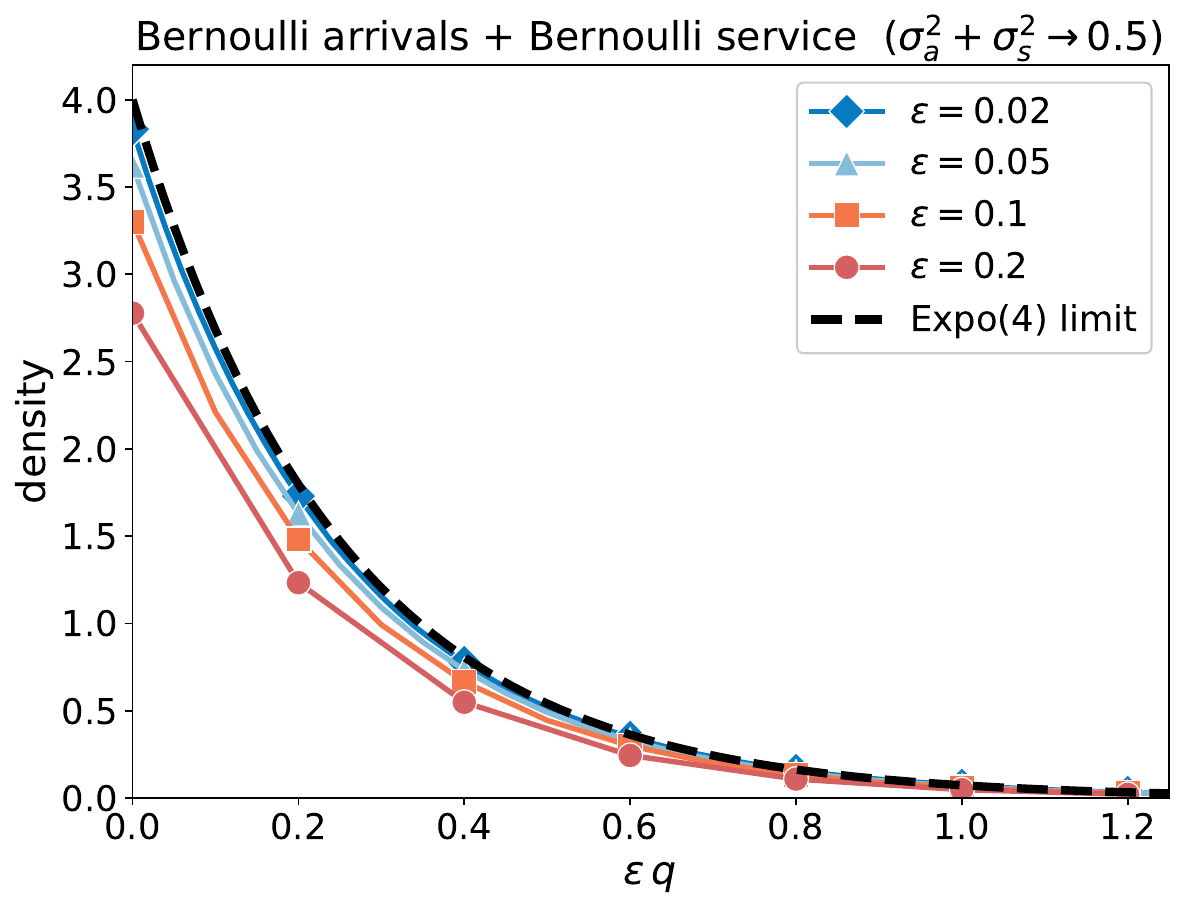}
\end{minipage}\hfill
\begin{minipage}[t]{0.49\textwidth}
\centering
\includegraphics[width=\textwidth]{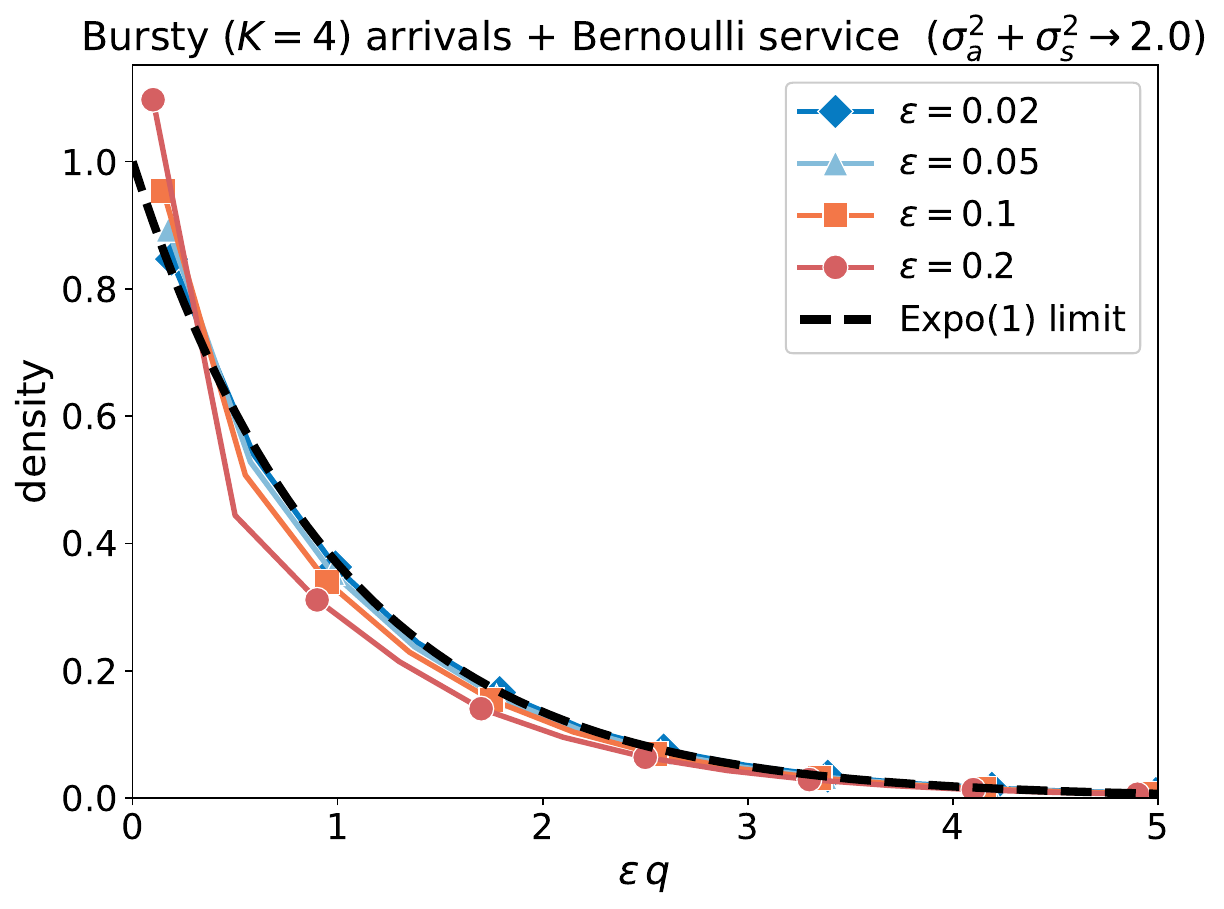}
\end{minipage}
\caption{Empirical density of the scaled queue length $\epsilon q$ in a
discrete-time single-server queue, for $\epsilon \in \{0.2, 0.1, 0.05, 0.02\}$,
compared against the heavy-traffic limit
$\mathrm{Expo}\!\left(2/(\sigma_a^2 + \sigma_s^2)\right)$ from
\eqref{eq:ht-result}. Left: Bernoulli arrivals and service with
$\lambda = 0.5 - \epsilon/2$, $\mu = 0.5 + \epsilon/2$
($\sigma_a^2 + \sigma_s^2 \to 0.5$). Right: bursty arrivals taking value
$4$ with probability $\lambda/4$ and value $0$ otherwise, with Bernoulli
service ($\sigma_a^2 + \sigma_s^2 \to 2.0$). Each curve is computed from a
simulation of length $T = 1.28 \times 10^{8}$ slots ($2.56 \times 10^{8}$
for $\epsilon = 0.02$), with the first half discarded as burn-in.}
\label{fig:ht-convergence}
\end{figure}

\subsection{Transform Method in Continuous Time: The $M/M/1$ Queue}\label{sec:ct-transform-method}
We now illustrate that the transform method adapts naturally to Markovian continuous time systems as well. Consider the $M/M/1$ queue where the times
between successive arrivals and services are exponentially distributed with rates $\lambda$ and $\mu$ respectively. As before, we denote the \emph{heavy-traffic parameter} $\epsilon \;:=\; \mu - \lambda \;>\; 0$. Let $q(t)$ denote the number of jobs in the system. Then, $\{q(t)\}_{t \ge 0}$ is a continuous-time Markov chain (CTMC).
To
emphasize the continuous-time setting, we track time with $t$ rather than $k$.
The three-step structure of the transform method is unchanged from
Section~\ref{sec:transform-method}; only the first step looks different, because the
role played by the one-step drift $e^{\theta\epsilon q(k+1)}-e^{\theta \epsilon q(k)}$ in
discrete time is now played by the \emph{generator} $\mathcal{G}$ of the
continuous-time process.

\subsubsection*{Step 1: Capturing the System Dynamics} The continuous-time analog of the one-step drift is the generator, which simply \new{considers all} the possible transitions of $q(t)$, weighting each by the rate at which it occurs.
Applying the
generator $\mathcal{G}$ to the test function $\varphi(q) = e^{\theta\epsilon q}$ gives
\begin{align}
    \mathcal{G} \varphi(q) &= \lambda\left(e^{\theta\epsilon (q+1)} - e^{\theta \epsilon q}\right) + \mu \mathbbm{1}_{\{q>0\}} \left(e^{\theta\epsilon (q-1)} - e^{\theta \epsilon q}\right), \label{eq:mm1generator}
\end{align}
as a customer arrives $(q \mapsto q+1)$ with rate $\lambda$ and service occurs $(q \mapsto q-1)$ whenever the system is non-empty $(q > 0)$ with rate $\mu$.
The discrete-time stationarity condition \new{$\E{\varphi(q(k+1))} = \E{\varphi(q(k))}$} becomes $\E{\mathcal{G} \varphi(q(t))}=0$.
Using this stationarity property, replacing $\mathbbm{1}_{\{q>0\}}=1-\mathbbm{1}_{\{q=0\}}$, dividing by $1-e^{-\theta \epsilon}$, and rearranging terms in \eqref{eq:mm1generator} yields
\begin{equation}\label{eq:ct-transform-before}
  \E{e^{\theta\epsilon q}\bigl(\mu\,e^{-\theta\epsilon} - \lambda\bigr)}
  \;=\;
  \new{e^{-\theta\epsilon}\,\mu\, \E{\mathbbm{1}_{\{q=0\}}}}.
\end{equation}
Equation~\eqref{eq:ct-transform-before} is the continuous-time counterpart of
the discrete-time identity~\eqref{eq:transform-equation-before-independence}.
The boundary term $\E{\mu\,\mathbbm{1}_{\{q=0\}}}$ on the right-hand side represents the rate of
service capacity lost to idleness, and plays the role of the unused service
$\E{e^{-\theta\epsilon u}}$ in the discrete time setting.
Note that setting $\theta=0$ in \eqref{eq:ct-transform-before} and \new{recalling} that $\E{\mathbbm{1}_{\{q=0\}}}=\mathbb{P}(q=0)$, we obtain
$
  \mu\,\mathbb{P}(q = 0)
  \;=\; \mu - \lambda
  \;=\; \epsilon
$.
In words, in steady state, the rate of lost service equals the excess service capacity.
Substituting $\mathbb{P}(q = 0)$ into~\eqref{eq:ct-transform-before}
and simplifying, we obtain the exact transform of the scaled queue length,
\begin{equation}\label{eq:ct-transform-equation}
  \E{e^{\theta\epsilon q}}
  \;=\;
  \frac{\epsilon}{\mu - \lambda\,e^{\theta\epsilon}},
\end{equation}
which is the continuous-time analog of the
\emph{transform equation}~\eqref{eq:transform-equation}. As in discrete time,
\eqref{eq:ct-transform-equation} is
\emph{exact}: it holds for every $\epsilon > 0$, not just in the
heavy-traffic limit. In fact, the right-hand side is precisely the MGF
of a geometric random variable with parameter $\rho = \lambda/\mu$ \new{(with the parameter $\theta$ replaced by $\theta \epsilon$). This is because it is well-known that the} $M/M/1$ queue has the geometric stationary distribution $\pi_i = (1-\rho)\rho^i$, from which
one could have computed the MGF directly as
$\mathbb{E}[e^{\theta\epsilon q}] = \sum_{i \ge 0} e^{\theta\epsilon i}
\pi_i$. Such a
closed-form-first strategy is the one taken by
\citet{cardinaels2024multi}, and in a similar spirit by
\citet{kingman1961_charfunction, boon2023heavy} when the pre-limit
transform itself is known in closed form.\footnote{\citet{boon2023heavy}
also use the term ``transform method.'' The distinction is that they
start from Pollaczek's formula for the transform and then pass to
heavy traffic, whereas the transform method in this tutorial works
directly with the queue dynamics.} \new{In contrast, the transform method is applicable even in systems where we do not have such closed form expressions, and we will see these in more detail in the next section.}

\subsubsection*{Steps 2 and 3: Heavy-Traffic Limit}
Steps 2 and 3 follow exactly as in the discrete-time queue of
Section~\ref{sec:transform}. Taylor-expanding the denominator
of~\eqref{eq:ct-transform-equation} for small $\epsilon$ and letting
$\epsilon \to 0$ yields $\epsilon q \stackrel{d}{\to}\mathrm{Expo}(1/\mu)$,
the continuous-time analog of the heavy-traffic result~\eqref{eq:ht-result}.
The only change from discrete time is \new{the parameter of the exponential. One way to see that this is the correct expression is to note that the mean queue length of the $M/M/1$ queue is $\lambda/(\mu-\lambda) = \lambda/\epsilon$, and so when multiplied by $\epsilon$, the limiting mean should be $\lambda \to \mu$.}

\subsection{Choice of Transform}\label{sec:choice-transform}
Before applying the transform method, one must decide which exponential test function to use.  Three natural candidates arise, each with distinct trade-offs.
\paragraph{Moment Generating Function (MGF).}  The MGF, $\mathbb{E}[e^{\theta\epsilon q}]$ for $\theta \in \mathbb{R}$, is the two-sided Laplace transform of the stationary distribution.  Its main advantage is that convergence of the MGF in an open interval around the origin implies convergence of \emph{all} moments, in addition to convergence in distribution.  However, one must first \emph{prove} that the MGF exists in a neighborhood of the origin, which requires establishing that the queue length has an exponentially decaying tail.  For the single-server queue, this follows from classical results (or can be established as part of the analysis as in \eqref{eq:mgf_finite_argument}), but in more complex systems this can be a nontrivial prerequisite.
\paragraph{Characteristic Function (CF).}  The CF, $\mathbb{E}[e^{j\omega\epsilon q}]$ for $\omega \in \mathbb{R}$ (where $j = \sqrt{-1}$), is the Fourier transform of the distribution.  Its key advantage is that it \emph{always exists}, \new{no} moment conditions are needed.  The price is that one \new{has to work} with complex \new{valued functions,} and convergence of characteristic functions (by L\'evy's continuity theorem, e.g., see \citet[Chapter 18]{williams1991probability}) gives only convergence in distribution, not convergence of moments.  Nevertheless, the CF is the natural choice when the MGF may not exist, or when the analysis is simpler in the Fourier domain (see Section~\ref{sec:matching-queue}).
\paragraph{One-Sided Laplace Transform.}  Since queue lengths are nonnegative, one can use $\mathbb{E}[e^{\theta\epsilon q}]$ restricted to $\theta \new{\leq} 0$.  This always exists because $e^{\theta\epsilon q} \le 1$ for $\theta \new{\leq} 0$ and $q \ge 0$.  The one-sided Laplace transform avoids the existence issues of the MGF but, like the CF, only yields convergence in distribution. However, like the MGF, one only works with real \new{valued functions}.

In summary, the MGF is the strongest tool when it can be shown to exist, the CF is the most universally applicable, and the one-sided Laplace transform offers a middle ground.

\subsection{From the Transform to Tail Bounds}\label{sec:tail-bounds}
A major advantage of working with the MGF is that it \new{provides bounds on the} tail probabilities. We illustrate this \new{in the case of} the discrete-time single-server queue\new{. Applying} Markov's inequality (the Chernoff bound) directly to the transform equation~\eqref{eq:transform-equation} \new{we} obtain a \emph{non-asymptotic} tail bound. Specifically, Markov's inequality gives
\begin{equation}\label{eq:tail-chernoff}
\mathbb{P}(\epsilon q > x)
\;=\; \mathbb{P}\bigl(e^{\theta\epsilon q} > e^{\theta x}\bigr)
\;\le\; e^{-\theta x}\,\mathbb{E}\bigl[e^{\theta\epsilon q}\bigr],
\end{equation}
which holds for every $\theta$ for which the MGF $\mathbb{E}\bigl[e^{\theta\epsilon q}\bigr]$ is finite. Substituting an appropriate second-order approximation of the terms in the transform equation~\eqref{eq:transform-equation} and optimizing over $\theta$, we obtain an upper bound of the form
\begin{equation} \label{eq:pre_limit_tail}
\mathbb{P}(\epsilon q > x) \;\le\; \underbrace{2ex/\sigma^2}_{\text{pre-exponent}} \;\cdot\; \underbrace{e^{-\theta_0 x}}_{\text{pre-limit tail}}, \quad \forall x > 2/\sigma^2,
\end{equation}
where $\sigma^2=\sigma_a^2+\sigma_s^2$ and $\theta_0 = \frac{2}{\sigma^2}\bigl(1-O(\epsilon)\bigr)$. We refer the readers to \citet{jhunjhun2024exponential} for details on steps of second order approximation required to obtain \eqref{eq:pre_limit_tail}. Taking the heavy-traffic limit $\epsilon \to 0$ in~\eqref{eq:pre_limit_tail} gives
\[
\lim_{\epsilon \to 0}\mathbb{P}(\epsilon q > x) \;\le\; \frac{2ex}{\sigma^2}\, e^{-\frac{2x}{\sigma^2}}, \quad \forall x > 2/\sigma^2,
\]
which recovers the correct tail decay rate of the limiting distribution. Note that this limit \emph{does not} recover the full limiting distribution of~\eqref{eq:ht-result}, owing to the \emph{pre-exponent} term $(2ex/\sigma^2)$ introduced by Markov's inequality. Nonetheless, the exponent of the \emph{pre-limit tail} \new{nearly matches} the $\mathrm{Expo}(2/\sigma^2)$ tail of the limiting distribution \new{(barring the $O(\epsilon)$ term in $\theta_0$)}. Indeed, since the pre-exponent is polynomial while the tail decays exponentially, the bound \new{provides a bound on} the large-deviation rate: $\lim_{x \to \infty} -\frac{1}{x}\log \mathbb{P}(\epsilon q > x) \geq \theta_0$. This provides a bridge between two classical regimes: \emph{heavy-traffic theory} (which studies the behavior as $\epsilon \to 0$) and \emph{large deviations} (which studies the exponential rate of decay of tail probabilities for fixed $\epsilon$). The transform method, by virtue of \new{working with} the MGF, naturally connects these two perspectives. Moreover, we emphasize that \eqref{eq:pre_limit_tail} is a \emph{pre-limit} result\new{, i.e., it} provides an explicit, computable guarantee on the tail probability at any finite $\epsilon > 0$. For practitioners, this means one can \new{guarantee} that a system meets a given service-level agreement (e.g., $\mathbb{P}(\text{delay} > x) \leq \eta$) at a specific, realistic load, without relying on the heavy-traffic approximation. \new{This approach is more formally presented} in~\citet{jhunjhun2024exponential}, where the authors also extend the analysis to load balancing under JSQ (see Section~\ref{sec:jsq-corrected}).

\subsection{Comparison with Other Methods}\label{sec:comparison}
The transform method is one of several ``direct'' methods for obtaining steady-state heavy-traffic results---methods that work directly with the stationary distribution of the pre-limit system.  We now compare it with three other direct methods: the drift method, Stein's method, and the Basic Adjoint Relationship (BAR) approach. Table~\ref{tab:comparison} summarizes the key differences among these four methods. \new{It is interesting to note that the drift method, transform method, and Stein's method are similar in spirit to their counterparts used to prove the central limit theorem. In contrast to the direct methods, classical work is based on a framework based on process level limits, and we briefly discuss that too.}
\subsubsection*{The Drift Method}
The drift method, as briefly outlined in Section~\ref{sec:kingman}, was introduced by~\citet{atilla} and uses \emph{polynomial} test functions.  To bound $\mathbb{E}[q^n]$, one applies the function $\varphi(q) = q^{n+1}$ to the queue evolution equation \eqref{eq:queue-evolution}, sets the steady-state drift to zero, and solves for $\mathbb{E}[q^n]$ in terms of lower moments.  This yields all moments by induction\new{. The main limitation of the drift method is that} each step requires the previous moments as input, making the algebra involved. The drift method applies broadly to Markovian queueing systems and has been used to study generalized switches by \citet{hurtado2020transform, hurtado2022logarithmic, hurtado2022heavy, Hurtado_JSQ_alpha_discrete}, input-queued switch by \citet{MagSri_SSY16_Switch}, parallel-server systems by \citet{zhong2019_process_flexibility, varma2021transportation}, load balancing by \citet{atilla, varma2025power}, bandwidth-sharing networks by \citet{Weina_bandwidth}, systems with Markov-modulated arrivals by \citet{grosof2024analysis}, and matching queues by \citet{varma2023dynamic}.
\subsubsection*{Stein's Method}
Stein's method \new{for stochastic networks was popularized by the works of}~\citet{braverman2017stein2} and \citet{Gurvich2014}. It characterizes the limiting distribution by working with \emph{an entire family of test functions} simultaneously.  The idea is \new{to work with the ``Stein's operator'' of the target distribution, and show that on this family of test functions, the generator of the queueing system is approximated by the Stein's operator.}  The strength of Stein's method is that it yields bounds on the \emph{Wasserstein distance} between the pre-limit and limiting distributions, thus characterizing the rate of convergence.  However, \new{the main challenge is determining the family of test functions. This family is obtained usually by solving the so-called Stein equation or the Poisson equation for the Stein operator. When one considers target distributions beyond the standard ones such as the Gaussian, exponential, etc., solving the Stein equation is challenging.} For recent developments of Stein's method for queueing systems, we refer the reader to \citet{braverman2018stein, gaunt2020stein, braverman2024high, chatterjee2026higher} and the references therein.
\subsubsection*{The BAR Approach}
\new{The BAR approach developed by~\citet{braverman_BAR} studies continuous time queueing systems operating under general interarrival and service time distributions (not necessarily exponential). Keeping track of such interarrival and service times is in general challenging. The BAR approach cleverly addresses those challenges while still working with exponential test functions. The transform method and the BAR method, thus, adopt the same approach of working with exponential test functions, but the BAR method can be more complex since it handles a more general class of arrival and service processes.} The BAR approach has been used to study Jackson networks by \citet{guang2026uniform}, load balancing by \citet{guang2025steady}, and multi-class queueing networks by \citet{braverman2025bar, Dai-2024-multiscale-HT-2, Dai-2026-multiscale-HT}.
\begin{table}
\centering
\TABLE{Comparison of direct methods for steady-state heavy-traffic analysis.
\label{tab:comparison}}{
\renewcommand{\arraystretch}{1.4}
\begin{tabular}{|p{7em}|p{9.5em}|p{8em}|p{8.5em}|p{8.5em}|}
\hline
 \textbf{Method} & \textbf{Drift Method} & \textbf{Transform} & \textbf{BAR} & \textbf{Stein} \\
\hline
\textbf{Test function} & Polynomial& Exponential & Exponential & Large class   \\
\hline
\textbf{Applicability} & Simple systems (mean, higher moments) & Markovian systems (simpler) & General distributions (more involved) & Requires target distribution \\
\hline
\textbf{What we get} & All pre-limit moments & Pre-limit tail and limiting distribution & Limiting distribution &  Rate of Convergence (Wasserstein) \\
\hline
\end{tabular}}{}
\end{table}
\subsubsection*{A Practitioner's Guide}
The four direct methods occupy distinct niches.
If \new{the goal is to bound the} \emph{moments}, the
drift method is the most \new{suitable}. If \new{the goal is to obtain} \emph{rate-of-convergence}
guarantees to the limiting distribution---i.e., quantitative bounds on
how close the finite-$\epsilon$ system is to its asymptotic
behavior---Stein's method is the standard tool. If \new{the goal is to obtain the pre-limit}
\emph{tail bounds}, or the entire \emph{steady-state distribution} \new{in the limit,}
the transform method is the natural choice and is what we develop in
the rest of this tutorial. If \new{one has to} handle general inter-arrival and service distributions, the BAR
approach extends the same exponential-test-function philosophy to
continuous time, with additional technical overhead. The methods are
\new{thus} complementary as much as they are competing.
\subsubsection*{Direct Methods versus Diffusion Limits}
It is instructive to contrast all four direct methods with the classical \emph{diffusion limit} approach.  This approach (e.g., see \citet{harrison1985brownian, whitt2002stochastic}) operates at the \emph{process level}\new{.} One shows that the sequence of scaled queue-length processes $\{q^{(\epsilon)}(k)\}$ converges weakly to a \new{diffusion process such as the} reflected Brownian motion (RBM) as $\epsilon \to 0$, and then characterizes the stationary distribution of the RBM.  This involves two limits\new{, viz.,} a time limit ($k \to \infty$) and a heavy-traffic limit ($\epsilon \to 0$).

\vspace{1em}
\noindent
\begin{minipage}[c]{0.55\textwidth}
Schematically, consider the following diagram. Let $q^{(\epsilon)}(k)$ denote the queue-length process with heavy-traffic parameter $\epsilon$ at time $k$, and let $q^{(\epsilon)}(\infty)$ denote its steady-state distribution. The goal is to compute the distribution \new{of} $q^{(\epsilon)}(\infty)$ \new{or its approximation as $\epsilon \to 0$.}
\end{minipage}%
\hfill
\begin{minipage}[c]{0.42\textwidth}
\[
\begin{array}{ccc}
q^{(\epsilon)}(k) & \xrightarrow{\;\epsilon \to 0\;} & q^{(0)}(k) \;\;\text{(RBM)}\\[8pt]
\Big\downarrow\; k \to \infty & & \Big\downarrow\; k \to \infty \\[8pt]
q^{(\epsilon)}(\infty) & \xrightarrow{\;\epsilon \to 0\;} & q^{(0)}(\infty)
\end{array}
\]
\end{minipage}
\vspace{1em}

The diffusion limits approach goes across the top and then down (e.g., see \citet{reiman1984open} for Jackson network and \citet{williams1998diffusion} for multi-class network)\new{, i.e.,} first take $\epsilon \to 0$ to get a diffusion process, then take $k \to \infty$ to get its stationary distribution. However, \new{since we are interested in the convergence of $q^{(\epsilon)}(\infty)$ to $q^{(0)}(\infty)$, one has to establish an interchange of limits, which is usually challenging to do} (\citet{gamarnik2006validity, budhiraja2009stationary}).

The direct methods (drift, Stein, BAR, transform) instead go down first and then across\new{, i.e.,} first take $k \to \infty$ to work with the stationary distribution at finite $\epsilon$, then take $\epsilon \to 0$.  The advantage of the direct route is that it \new{is usually simpler,} avoids the interchange-of-limits issue and can often provide \emph{pre-limit bounds} that hold for every $\epsilon > 0$.  The transform method, in particular, produces the exact transform equation at finite $\epsilon$, from which both the limiting distribution and non-asymptotic tail bounds can be extracted.

\subsection{Summary and Preview}\label{sec:preview}
The transform method, as illustrated here for the single-server queue, consists of three steps:
\begin{enumerate}
    \item \textbf{Capture the system dynamics} via the steady-state drift of an exponential test function, exploiting the complementarity condition to derive an exact transform equation
    \item Perform a \textbf{second-order approximation} of the transform equation for small $\epsilon$
    \item \textbf{Solve the resulting functional equation} in the heavy-traffic limit
\end{enumerate}
For the single-server queue, Step~3 is trivial (just take $\epsilon \to 0$), but as we will see in subsequent sections \new{that} it becomes significantly more involved in richer settings. The power of the method lies in its ability to produce the full distributional characterization directly from the queue dynamics, without requiring knowledge of the pre-limit stationary distribution \new{and without having to establish a} process-level diffusion limit.  Moreover, the exact pre-limit transform equation naturally yields non-asymptotic tail bounds, connecting the heavy-traffic and large deviations regimes.

\new{The method has been extended along two natural dimensions, where the three-step recipe is universal, but the difficulty shifts across settings. In the first dimension, presented in Section~\ref{sec:single-server-variants} and summarized in Table~\ref{tab:single-server}, we work with the single-server model but enrich it with features such as customer abandonments, Markov-modulated arrivals, and state-dependent arrivals. In these models, the main challenge is that the richer system dynamics lead to more involved functional equations, requiring new techniques to establish the transform in closed form. The second dimension, presented in Section~\ref{sec:networks} and summarized in Table~\ref{tab:networks}, moves from a single queue to \emph{stochastic networks} involving multiple interacting queues. Here, the main new ingredient is \emph{state space collapse} (SSC). The high-dimensional queue-length vector concentrates near a lower-dimensional subspace in heavy traffic, and the transform method must be applied to the process projected onto the lower-dimensional subspace. In these networks, both Steps~1 and~3 become substantially harder (establishing state space collapse and solving multi-dimensional functional equations, respectively).}

\begin{table}
\centering
\TABLE{Transform method applied to single-server queue variants (Section~\ref{sec:single-server-variants}).
\label{tab:single-server}}{
\renewcommand{\arraystretch}{1.4}
\begin{tabular}{|>{\raggedright\arraybackslash}p{7.2em}
               |>{\raggedright\arraybackslash}p{11em}
               |>{\raggedright\arraybackslash}p{13em}
               |>{\raggedright\arraybackslash}p{10.8em}|}
\hline
\textbf{Variant} & \textbf{New Challenge} & \textbf{New Technique (Step 3)} & \textbf{Limiting distribution} \\
\hline
State-dependent  & Arrival/Service rate depends on queue length & Solve functional equation via inverse Fourier transforms & Gibbs (depending on state-dependent control) \\
\hline
Markov-modulated & Arrival/Service rate is driven by a Markov chain & Poisson equation for the modulating chain & Exponential (with effective variance) \\
\hline
Abandonments & Customers abandon at rate $\gamma$ & Establish and solve an ODE in terms of the MGF & Exponential / Truncated Gaussian depending on $\gamma$ \\
\hline
Pre-limit (finite $\epsilon$) tail bounds  & Non-asymptotic pre-limit analysis & Tight approximation to obtain multiplicative error & Exponential decay with explicit rate $\eta(\epsilon)$ \\
\hline
\end{tabular}}{}
\end{table}

\begin{table}[ht]
\TABLE{Brief Summary of Results using Transform Method applied to Multi-Dimensional Stochastic Networks (Section~\ref{sec:networks}).
\label{tab:networks}}{
\renewcommand{\arraystretch}{1.4}
\begin{tabular}{|>{\centering\arraybackslash}p{3cm}
               |>{\centering\arraybackslash}p{3.8cm}
               |>{\centering\arraybackslash}p{4.4cm}
               |>{\centering\arraybackslash}p{4.1cm}|}
\hline
 & \textbf{Load Balancing: Many Server}
 & \textbf{Load Balancing: Abandonment}
 & \textbf{Input-Queued-Switch: Multiple Bottleneck} \\
\hline
\textbf{Regime}
 & Many-Server heavy traffic
 & Critically Loaded and Heavily Overloaded
 & Classic Heavy Traffic \\
\hline
\textbf{SSC result (Step 1)}
 & A tight $O(1)$ bound on queue imbalance
 & Under \new{a} combination of load balancing \& abandonment
 & Under a large class of low-complexity algos \\
\hline
\textbf{Functional Equation (Step 3)}
 & Non-asymptotic MGF bound
 & Differential equation in terms of MGF
 & Implicit (or joint) Eq.\ with multiple variables \\
\hline
\textbf{Tail Bound / Distribution}
 & Expo. distribution and sharp tail bound
 & Normal or truncated normal distribution
 & Non-linear combination of i.i.d.\ exponentials \\
\hline
\end{tabular}}{}
\end{table}

We close with a brief note on limitations\new{.} The transform method is most naturally suited
to Markovian dynamics, \new{i.e., discrete time systems with independent or Markovian arrivals and continuous time systems with exponential inter-arrival and service times. When one has continuous time systems with general distributions, one has to use the} BAR method
(Section~\ref{sec:comparison})\new{. In} multi-bottleneck networks, Step~3
produces multi-dimensional functional equations whose uniqueness can be
hard to establish, as we discuss in Section~\ref{sec:networks}.

\section{Single-Server Queue Variants}\label{sec:single-server-variants}
This section develops three variants of the single-server queue, each illustrating a distinct challenge that arises \new{due to the} richer dynamics\new{, and how the basic three step procedure of the transform method can be adapted.}

\subsection{The Single-Server Queue with Abandonment}\label{sec:abandonment}
In many service systems, customers grow impatient while waiting in the queue and \emph{abandon} before being served.  \new{For example,} callers hang up in call centers, patients leave emergency departments, and riders cancel ride-hail requests\new{. Abandonment} fundamentally changes the queue-length behavior. In this section, we apply the transform method to a variant of the classical single-server queue that incorporates customer abandonment. This model has also been studied using the diffusion approximation approach by \citet{ward2003diffusion, ward2005diffusion}. We also refer the reader to \citet{garnett2002designing, zeltyn2005call, reed2012hazard} for diffusion approximation of \new{abandonment queues} and its generalizations, and \citet{ward2012asymptotic} for a survey.
\subsubsection*{What's new here (Preview).} In this section, we
will see that \new{when we use the transform method,} the MGF no
longer has the closed form \eqref{eq:approx-mgf}\new{. Instead, we solve an ordinary differential equation to obtain the limiting MGF.}

\subsubsection*{Model Description}
Consider the discrete-time single-server queue from Section~\ref{sec:transform}, augmented with customer abandonment.  At each time slot $k$, arrivals $a(k)$ are i.i.d.\ with mean $\lambda$ and variance $\sigma_a^2$, and the server can serve $s(k)$ customers with $\mathbb{E}[s(k)] = \mu$ and $\mathrm{Var}[s(k)] = \sigma_s^2$.  The \new{main difference now} is that each customer present in the queue at time $k$ independently abandons with probability $\gamma \in (0,1)$.  The total number of abandonments is therefore
\begin{equation*}
d(k) \;\sim\; \mathrm{Bin}\bigl( q(k),\, \gamma\bigr),
\end{equation*}
a binomial random variable with parameters $q(k)$ and $\gamma$.  The queue length evolves according to
\begin{align}
\label{eq:queue-abandon}
q(k+1) = \; \bigl[q(k) + a(k) - s(k) - d(k)\bigr]^+ = \; q(k) + a(k) - s(k) - d(k) + u(k),
\end{align}
where we introduce the unused service $u(k) \ge 0$ as before,
with the complementarity condition $q(k+1) \cdot u(k) = 0$. A crucial difference from the classical queue is that the system with abandonment is \emph{always stable}, for any arrival rate $\lambda$ and any \new{positive} abandonment probability $\gamma > 0$.  The expected number of abandonments grows linearly with the queue length (i.e., $\mathbb{E}[d(k)] = \gamma q(k)$), as a result, \new{as the arrival rate increases,} the decrease in the queue length due to abandonment dominates the increase due to arrivals, leading to stability. Consequently, there is no requirement that $\lambda < \mu$; the system can be \emph{overloaded} ($\lambda > \mu$) and \new{it is still stable, and has a} unique stationary distribution.

\subsubsection*{Heavy-Traffic Scaling and Regimes}
Without abandonment, the heavy-traffic parameter is $\epsilon = \mu - \lambda$ and the natural scaling is $\epsilon q$.  With abandonment, the scaling changes because abandonment provides a \new{negative drift} of order $\gamma q$, which competes \new{with the drift $-\epsilon = -(\mu - \lambda)$}. Thus, the rate at which $\epsilon$ and $\gamma$ tend to zero governs the limiting distribution of the appropriately scaled queue length. More precisely, depending on the ratio $C := -\epsilon / \sqrt{\gamma}$, we observe three distinct regimes, and our results are \new{presented in} Table~\ref{tbl: dist}.
\begin{table}[ht]
\TABLE{Limiting steady-state distribution of the scaled queue length vector in the three regimes, as the abandonment probability $\gamma$ \new{and/or the heavy traffic parameter $\epsilon$} goes to zero. Here, HT stands for heavy traffic.
      \label{tbl: dist}}{
     \small\addtolength{\tabcolsep}{-3pt}
     \begin{tabular}{| P{3cm} | P{9cm} | P{3.2cm} |}
     \hline
     \rule{0pt}{11pt}
     Classic-HT
     & Critical-HT   &  Heavily Overloaded   \\
     \hline
     \rule{0pt}{12pt}
     $C\rightarrow - \infty, \epsilon \rightarrow 0$ & (a) $C< 0$, (b) $C=0$, (c) $C>0$ & $C\rightarrow\infty$ \\
     \hline
      $\epsilon q \stackrel{d}{\rightarrow}$ & $\sqrt{\gamma}q \stackrel{d}{\rightarrow}$ & $\sqrt{\gamma}(q-\mathbb E[q]) \stackrel{d}{\rightarrow}$ \\
      \hline
      \rule{0pt}{45pt}
     \includegraphics[width=0.16\textwidth]{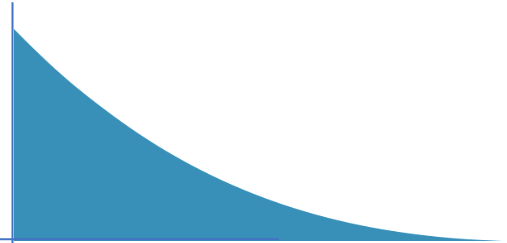}
      & \includegraphics[width=0.16\textwidth]{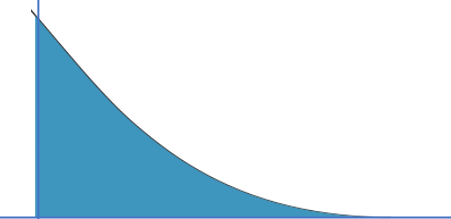} \includegraphics[width=0.16\textwidth]{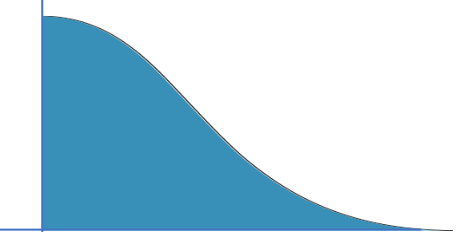} \includegraphics[width=0.16\textwidth]{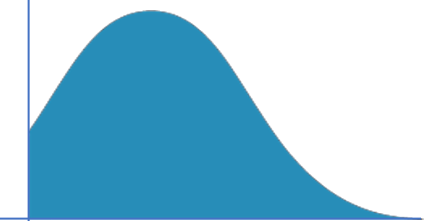}
      &
      \includegraphics[width=0.16\textwidth]{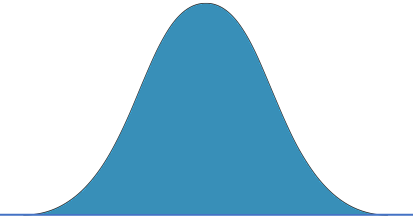}
      \\ $\text{Exponential}$ & $\text{Truncated-Normal}$ & $\text{Normal}$ \\ \hline
      \end{tabular}}{We refer the reader to \citet{jhunjhunwala2026jsqa} for formal statement of these results.}
\end{table}
\new{These} results formalize a phase transition in the limiting scaled queue length from exponential to a Normal distribution as the system moves from underloaded $(C < 0)$ to overloaded $(C > 0)$. The Critical-HT regime in the middle interpolates between the exponential (when $C \to -\infty$) and the Gaussian (when $C \to +\infty$)\new{, via a truncated Gaussian}. We focus our analysis on \new{the critical-HT} regime.
\subsubsection*{Step 1: Capturing the System Dynamics}
As in the classical case, we apply the exponential test function $e^{\sqrt{\gamma}\theta q}$ (note the $\sqrt{\gamma}$ scaling) and use the complementarity condition to obtain:
\begin{equation}\label{eq:abandon-transform}
\mathbb{E}\bigl[e^{\sqrt{\gamma}\theta(a - s)} - 1\bigr]\, \mathbb{E}\bigl[e^{\sqrt{\gamma}\theta  q}\bigr] \;+\; \mathbb{E}\bigl[(e^{-\sqrt{\gamma} \theta  d} - 1)\, e^{\sqrt{\gamma} \theta  q}\bigr] \;+\; \bigl(1 - \mathbb{E}\bigl[e^{-\sqrt{\gamma} \theta  u}\bigr]\bigr) \;=\; o(\gamma).
\end{equation}
Without abandonment ($d = 0$), this reduces to the classical transform equation of Section~\ref{sec:transform-method}.  The \emph{abandonment term} $\mathbb{E}[(e^{-\sqrt{\gamma} \theta  d} - 1)\, e^{\sqrt{\gamma} \theta  q}]$ is new\new{, and will be handled in the next step. Note that since we are working with the MGF, we need to show that it exists for $\theta$ in an open interval around the origin. This is done using novel Lyapunov drift arguments in \citet{jhunjhunwala2026jsqa}, overcoming the challenge due to overloaded arrivals, i.e., $\lambda > \mu$.}
\subsubsection*{Step 2: Second-Order Approximation}
Each of the three terms in~\eqref{eq:abandon-transform} is now approximated for small $\gamma$. In \eqref{eq:abandon-transform}, for the arrival--service term (the first term) and the idleness term (the third term), we use a similar approximation as in Section~\ref{sec:transform-method}, that is
\begin{align*}
\frac{1}{\gamma}\mathbb{E}\bigl[e^{\sqrt{\gamma}\,\theta(a-s)} - 1\bigr] \;\approx\; C\theta + \tfrac{1}{2}\theta^2\sigma^2, && \frac{1}{\gamma}\bigl(1 - \mathbb{E}[e^{-\sqrt{\gamma} \theta  u}]\bigr) \approx \frac{\theta}{\sqrt{\gamma}}\,\mathbb{E}[u].
\end{align*}
The main difference here lies in the approximation of the abandonment term. Since $d \sim \mathrm{Bin}(q, \gamma)$, we have $\mathbb{E}[d \mid q] = \gamma q$, so,
\[
\mathbb{E}\bigl[(e^{-\sqrt{\gamma} \theta  d} - 1)\, e^{\sqrt{\gamma} \theta  q}\bigr] \;\approx\; -\mathbb{E}\bigl[\sqrt{\gamma} \theta  d\, e^{\sqrt{\gamma} \theta  q}\bigr] \;=\; -\sqrt{\gamma} \theta \mathbb{E}\bigl[\gamma q\, e^{\sqrt{\gamma} \theta  q}\bigr] \;=\; -\gamma\theta\,\frac{d}{d\theta}\mathbb{E}\bigl[e^{\sqrt{\gamma} \theta  q}\bigr],
\]
where the second equality used the tower property, and in the last step we used \new{$\frac{d}{d\theta } \mathbb{E}[e^{\sqrt{\gamma} \theta  q}] =\sqrt{\gamma}\, \mathbb{E}[q\, e^{\sqrt{\gamma} \theta  q}]$}. Denoting $M(\theta) := \lim_{\gamma \to 0} \mathbb{E}[e^{\sqrt{\gamma}\theta  q}] $ \new{(after showing} that the limit exists) collecting all terms, and dividing by $\gamma \theta$ in \eqref{eq:abandon-transform}, we arrive at the \emph{differential equation}
\begin{equation}
\label{eq:mgf-ode}
\left(C + \tfrac{1}{2}\theta \sigma^2\right) M(\theta) \;-\; \frac{dM(\theta)}{d\theta} \;+\; \frac{1}{\sqrt{\gamma}}\,\mathbb{E}[u] \;=\; o(1).
\end{equation}
\subsubsection*{Step 3: Solving the Differential Equation}
As we take $\gamma\rightarrow0$, \new{the right hand side of} \eqref{eq:mgf-ode} \new{vanishes, and we get a} first-order linear ODE in terms of $M(\theta)$, where $\lim_{\gamma\rightarrow 0 }\frac{1}{\sqrt{\gamma}}\,\mathbb{E}[u]$ acts as an unknown constant. We solve for it using the boundary conditions $M(0) = 1$ (since $M$ is an MGF) and $M(\theta) \to 0$ as $\theta \to -\infty$ (since the limiting distribution is supported on $[0, \infty)$ and the MGF must decay for negative $\theta$). This is in contrast with the single-server queue (without abandonment), where the boundary conditions are not required. Now, the solution of ODE in \eqref{eq:mgf-ode}, subject to the boundary conditions, yields the MGF of a \emph{truncated normal} distribution\new{,}
\begin{equation}\label{eq:truncated-normal}
\sqrt{\gamma} q \;\xrightarrow{d}\; \mathrm{TruncNormal}(C,\, \sigma^2) \qquad \text{as } \gamma \to 0,
\end{equation}
where $\mathrm{TruncNormal}(C, \sigma^2)$ denotes a normal random variable with mean $C$ and variance $\sigma^2$, truncated (and normalized) to the nonnegative half-line $[0, \infty)$. The formal proof with rigorous justification of all the above approximations can be found in~\citet{jhunjhunwala2026jsqa}.

\subsection{Markov-Modulated Arrivals and the Poisson Equation}\label{sec:markov-modulated}
In many real-world service systems, arrivals and service rates are not \new{stationary. In call centers, ride hailing systems, data centers, etc., demand patterns shift with time of the day, weather, and special events. All of these can be captured by a Markov chain, which can be modeled to modulate the arrival and/or service rates. Traffic in communication networks is also known to be bursty, and} is naturally modeled by Markov-modulated processes (\citet{fischer1993markov}). This section deals with a Markov-modulated single-server queue in continuous time, extending the $M/M/1$ queue analysis in Section~\ref{sec:ct-transform-method}.
\subsubsection*{What's new here (Preview).} The arrival and service rates have temporal correlations as they are driven by an exogenous Markov chain, so the queue length
and the environment evolve jointly. The \new{main challenge in the transform method now is that the}
transform equation no longer factors as in
Section~\ref{sec:transform}. We address this with the \new{use of the}
\emph{Poisson equation}, a classical object from Markov chain theory
(see, e.g., \citet{glynn2023solution, makowski2002poisson,
meyn2009markov}) that \new{enables us to handle the temporal correlations, eventually leading to an}
\emph{effective variance} replacing $\sigma_a^2 + \sigma_s^2$ in the
heavy-traffic limit.
\begin{DTonly}
%% ----- Model Description -----
\subsubsection{Model Description}
Consider a single-server queue operating in a random environment.
Let $\{X(k)\}_{k\ge 0}$ be an ergodic Markov chain on a countable state space
$\mathcal{X}$ with transition matrix $P = [P_{ij}]_{i,j\in\mathcal{X}}$ and stationary distribution
$\pi$.  Conditioned on the environment state $X(k) = i$, the number of arrivals
$a(X(k), k)$ and the service capacity $s(X(k), k)$ at time $k$ are independent (given $X(k)$) random variables
with means $\lambda(i)$, $\mu(i)$ and variances $\sigma^2_{a}(i)$, $\sigma^2_{s}(i)$, respectively and finite support.   The long-run
average rates are
\[
  \lambda \;=\; \sum_{i\in\mathcal{X}} \pi_i\,\lambda_i, \qquad
  \mu \;=\; \sum_{i\in\mathcal{X}} \pi_i\,\mu_i,
\]
with $\epsilon := \mu - \lambda > 0$ governing the proximity to heavy traffic. The queue evolves according to the Lindley-type recursion
\begin{equation}\label{eq:mm-lindley}
  q(k+1) \;=\; \bigl[q(k) + a\bigl(X(k), k\bigr) - s\bigl(X(k), k\bigr)\bigr]^+,
\end{equation}
We also assume that $\sum_{i \in \mathcal{X}} \pi_i (\lambda_i-\mu_i)^2 = O(\epsilon)$. This is exactly the single-server queue of Section~2, except that the arrival and service
distributions \emph{change from slot to slot} according to the state of the modulating
chain~$X$. As in Section~\ref{sec:model}, the queue evolution can be written as
\begin{equation}\label{eq:mm-evolution}
  q(k+1) \;=\; q(k) + a\left(X(k), k\right) - s\left(X(k), k\right) + u(X(k), k),
\end{equation}
where $u(X(k), k) \ge 0$ is the unused service and the complementarity condition
$q(k+1)\cdot u(X(k), k) = 0$ holds.
This model is the discrete time version of \citet{HL-Gro-2026-Markov-Modulated}. Working in discrete time does not conceptually change the proofs; we provide more details of the intermediate steps whenever there are differences from the continuous-time system.
%% ----- Why standard i.i.d.\ analysis fails -----
\subsubsection{Why Standard i.i.d.\ Analysis Fails}
In the i.i.d.\ setting, the key factorization in steady state
$\E{e^{\theta\epsilon(a-s)}\,e^{\theta\epsilon q}} = \E{e^{\theta\epsilon(a-s)}}\,\E{e^{\theta\epsilon q}}$
relies on the independence of $a(k) - s(k)$ from $q(k)$.
With Markov modulation this breaks down: since $a(k)$ and $s(k)$ depend on $X(k)$,
and $X(k)$ is correlated with $q(k)$ through the joint chain $(q(k), X(k))$, the drift
calculation produces cross-terms $\E{e^{-s\epsilon q}\,f(X)}$ that cannot be factored.
In other words, we cannot solve for the MGF in~\eqref{eq:transform-equation-before-independence} and need a new toolbox.
%% ----- The Poisson Equation -----
\subsubsection{The Poisson Equation}
The Poisson equation is the key tool we will rely on.
\begin{definition}[Poisson equation]\label{def:poisson}
Given a function $f : \mathcal{X} \to \mathbb{R}$ with $\bar{f} := \bE_\pi\left[f(X)\right] = \sum_{i\in\mathcal{X}} \pi_i f(i)$,
the \emph{Poisson equation} for $f$ asks for a function $V_f : \mathcal{X} \to \mathbb{R}$ satisfying
\begin{equation}\label{eq:poisson}
  V_f(i) - \sum_{j\in\mathcal{X}} P_{ij}\, V_f(j) \;=\; f(i) - \bar{f},
  \qquad \forall\, i \in \mathcal{X},
\end{equation}
where we define $V_f(i^+) := \sum_{j\in\mathcal{X}} P_{ij}\, V_f(j).$
Equivalently $(I - P)\,V_f = f - \bar{f}$.
\end{definition}
The solution $V_f$ is unique up to additive constants and, intuitively, characterizes
the total expected future deviation of $f(X)$ from its mean, starting from state~$i$. The central role of the Poisson equation is captured by the following lemma, whose proof
we defer to the appendix.
\begin{lemma}[Poisson lemma]\label{lem:poisson}
Let $f : \mathcal{X} \to \mathbb{R}$ be such that the Poisson equation for $f$ has a solution
$V_f : \mathcal{X} \to \mathbb{R}$ with $\E{|V_f(X)|^{1+\xi}} < \infty$ for some $\xi > 0$.
Then, for all $\theta < 0$,
\begin{equation}\label{eq:poisson-lemma}
  \E{e^{\theta\epsilon q}\, f(X)}
  \;=\;
  \E{e^{\theta\epsilon q}}\,\bar{f}
  \;-\;
  \theta\epsilon\,\E{e^{\theta\epsilon q}\, V_f(X^+)\,(\mu(X) - \lambda(X))}
  \;+\;
  O\!\Bigl(\epsilon^{1+\frac{\xi}{1+\xi}}\Bigr),
\end{equation}
where we omit the time dependence on $k$ for steady-state random variables.
\end{lemma}
The first term on the right-hand side is the ``decoupled'' approximation (as if $f(X)$ and $q$ were
independent), and the second is a correction of order~$\epsilon$ that captures the
covariance between the environment and the queue through the Poisson solution~$V_f$.
%% ----- Step 1: Transform Equation -----
\subsubsection{Extension of the Transform Method}
In this case, we work with the one-sided Laplace transform variable $\theta < 0$
and set the drift of $\varphi(q,X) = e^{\theta\epsilon q}$ to zero in steady state.
\subsubsection*{Step 1: Capturing the System Dynamics}
Using the complementarity condition $(e^{\theta\epsilon q(k+1)} - 1)(e^{-\theta\epsilon u(k)} - 1) = 0$, taking expectations in steady state over the joint distribution of $(q, X)$ and rearranging terms we obtain the transform equation:
\begin{equation}\label{eq:mm-transform}
  \E{e^{\theta\epsilon q}\bigl(1 - e^{\theta\epsilon(a(X) - s(X))}\bigr)}
  \;=\; 1 - \E{e^{-\theta\epsilon u}}.
\end{equation}
Note the
dependence on $X$ for $a(X) - s(X)$: the arrival and service distributions depend on the
state of the environment, so the left-hand side involves cross-terms between the queue
length and functions of the Markov chain.
We now expand the transform equation~\eqref{eq:mm-transform} to second order
in~$\epsilon$.
\subsubsection*{Step 2: Second Order Approximation}
Similar to \eqref{eq:unused_service} for the single-server queue setting, we can simplify the right-hand side of~\eqref{eq:mm-transform} to get
\begin{equation}\label{eq:mm-rhs}
  1 - \E{e^{-\theta\epsilon u}} \;=\; \theta\epsilon^2 + O(\epsilon^3).
\end{equation}
Next, Taylor-expanding $e^{\theta\epsilon(a(X) - s(X))}$ on the left-hand side
of~\eqref{eq:mm-transform} and using the tower property of expectation, we get
\begin{align}
  &\E{e^{\theta\epsilon q}\bigl(1 - e^{\theta\epsilon(a(X) - s(X))}\bigr)} \nonumber
  \\
  & \quad = -\theta\epsilon\,\underbrace{\E{e^{\theta\epsilon q}(\lambda(X) - \mu(X))}}_{(\star)}
  \;-\; \frac{\theta^2\epsilon^2}{2}\,
        \underbrace{\E{e^{\theta\epsilon q}\left(\sigma^2_{a}(X) + \sigma^2_{s}(X)\right)}}_{(\star\star)}
  \;+\; O(\epsilon^3), \label{eq:mm-lhs-expanded}
\end{align}
where, for $(\star\star)$, we used the tower property to get
\begin{align*}
    \E{e^{\theta \epsilon q} (a(X)-s(X))^2} &= \E{e^{\theta \epsilon q} \E{(a(X)-s(X))^2 | q, X}} \\
    &= \E{e^{\theta \epsilon q}\left(\sigma^2_{a}(X) + \sigma^2_{s}(X)\right)} + \E{e^{\theta \epsilon q} (\lambda(X)-\mu(X))^2} \\
    &= \E{e^{\theta \epsilon q}\left(\sigma^2_{a}(X) + \sigma^2_{s}(X)\right)} + O(\epsilon),
\end{align*}
where the last equality follows as $\theta < 0, q \geq 0$ and $\E{(\lambda(X)-\mu(X))^2} = O(\epsilon)$. The terms $(\star)$ and $(\star\star)$ involve cross-terms between $e^{\theta\epsilon q}$
and functions of the environment state~$i$.  This is precisely where the Poisson equation
machinery is needed.
\subsubsection*{Step 3: Simplifying via the Poisson Equation}
We start by simplifying $(\star)$ via the Poisson lemma. Apply Lemma~\ref{lem:poisson} with $h(X) = \mu(X) - \lambda(X)$, so $\bar{h} = \epsilon$ (assuming $\E{|V_{h}(X)|^{1+\xi}} < \infty$):
\begin{equation}\label{eq:star}
  \E{e^{\theta\epsilon q}\,(\mu(X) - \lambda(X))}
  \;=\;
  \E{e^{\theta\epsilon q}}\,\epsilon
  \;-\;
  \theta\epsilon\,\E{e^{\theta\epsilon q}\,V_h(X^+)\,(\mu(X) - \lambda(X))}
  \;+\;
  O\!\Bigl(\epsilon^{1+\frac{\xi}{1+\xi}}\Bigr).
\end{equation}
Now define
\begin{equation}\label{eq:def-khat}
  \kappa(X) \;:=\; V_{h}(X^+)\,(\mu(X) - \lambda(X)),
\end{equation}
and apply Lemma~\ref{lem:poisson} a second time with $f = \kappa$ (assuming $\E{|V_{\kappa}(X)|^{1+\xi}} < \infty$).
Since the entire expression $\E{e^{\theta\epsilon q}\kappa(X)}$ is already multiplied
by~$\theta\epsilon$ in~\eqref{eq:star}, we only need the leading-order term:
\begin{equation}\label{eq:khat-leading}
  \E{e^{\theta\epsilon q}\,\kappa(X)}
  \;=\;
  \E{e^{\theta\epsilon q}}\,\bE_\pi\left[\kappa(X)\right]
  \;+\; O(\epsilon).
\end{equation}
Substituting~\eqref{eq:khat-leading} back
into~\eqref{eq:star}:
\begin{equation}
\label{eq:star-final}
  \E{e^{\theta\epsilon q}\,(\mu(X) - \lambda(X))} = \E{e^{\theta\epsilon q}}
  \Bigl(\epsilon - \theta\epsilon\,\bE_\pi\left[\kappa(X)\right]\Bigr)
  + O\Bigl(\epsilon^{1+\frac{\xi}{1+\xi}}\Bigr).
\end{equation}
The quantity $\bE_\pi\left[\kappa(X)\right]$ encodes the
effect of Markov modulation on the queue through the Poisson equation solution
$V_{h}$.  When the environment is i.i.d.\ ($P_{ij} = \pi_j$), $V_{h}$ is
constant, $\bE_\pi[\kappa(X)] = 0$, and the correction vanishes. Next, we compute $(\star\star)$ via the Poisson lemma. Applying Lemma~\ref{lem:poisson} with $f(X) = \sigma^2_{a}(X) + \sigma^2_{s}(X)$ (assuming $\E{|V_{f}(X)|^{1+\xi}} < \infty$), we obtain:
\[
  \E{e^{\theta\epsilon q}\,\left(\sigma^2_{a}(X) + \sigma^2_{s}(X)\right)}
  \;=\;
  \E{e^{\theta\epsilon q}}\,\mathbb{E}_{\pi}\left[\sigma^2_{a}(X) + \sigma^2_{s}(X)\right]
  \;+\; O(\epsilon),
\]
where $\mathbb{E}_{\pi}\left[\sigma^2_{a}(X) + \sigma^2_{s}(X)\right]
  = \sum_i \pi_i\,(\sigma^2_{a,i} + \sigma^2_{s,i})$.
Since this term is already multiplied by $\theta^2\epsilon^2/2$
in~\eqref{eq:mm-lhs-expanded}, the $O(\epsilon)$ correction is absorbed into the overall
$O(\epsilon^3)$ error. Combining~\eqref{eq:mm-rhs} and~\eqref{eq:mm-lhs-expanded} with the Poisson lemma
computations, and noting that the correction from $(\star)$ contributes the effective
variance term,
and rearranging:
\begin{equation}\label{eq:mm-mgf-prelimit}
  \E{e^{\theta\epsilon q}}
  \;=\;
  \frac{1 + o(1)}{1 - \frac{1}{2} \theta\Bigl(
    \mathbb{E}_{\pi}\left[\sigma^2_{a}(X) + \sigma^2_{s}(X)\right]
    + 2\mathbb{E}_{\pi}\left[\kappa(X)\right]\Bigr) + o(1)},
\end{equation}
Taking $\epsilon \to 0$ in~\eqref{eq:mm-mgf-prelimit}, we recognize the right-hand
side as the Laplace transform of an exponential random variable:
\begin{equation}\label{eq:mm-limit}
  \epsilon q
  \;\xrightarrow{d}\;
  \mathrm{Expo}\!\left(\frac{2}{\sigma^2_{\mathrm{eff}}}\right),
\end{equation}
where $q$ denotes the steady-state queue length and the \emph{effective variance}
is defined by
\begin{equation}\label{eq:eff-var}
  \sigma^2_{\mathrm{eff}}
  =
  \lim_{\epsilon \to 0} \mathbb{E}_{\pi}\left[\sigma^2_{a}(X) + \sigma^2_{s}(X)\right] + 2\mathbb{E}_{\pi}\left[\kappa(X)\right],
\end{equation}
assuming the limit exists. This has the same \emph{form} as the classical result
$\epsilon q \to \mathrm{Expo}(2/\sigma^2)$ from Section~2, but with the i.i.d.\ variance
$\sigma^2 = \sigma^2_a + \sigma^2_s$ replaced by the effective variance
$\sigma^2_{\mathrm{eff}}$ that accounts for the temporal correlations introduced by the
modulating chain.
\paragraph{Independence from the mixing time.}
A striking feature of~\eqref{eq:mm-limit} is that its validity requires only finiteness of
certain moments of the Poisson solutions $V_{h}$, $V_{\kappa}$, and
$V_{\hat{\ell}}$---conditions that depend on the \emph{structure} of the modulating chain
but \emph{not on its mixing time}.  Classical approaches to Markov-modulated queues
typically require the chain to mix quickly relative to the heavy-traffic timescale; the
Poisson equation sidesteps this entirely.
\paragraph{Summary.}
Markov modulation breaks the independence that underlies the classical transform method,
creating cross-terms $\E{e^{\theta\epsilon q}\,f(X)}$ that resist direct computation.
The Poisson equation decomposes each such cross-term into a product of marginals plus a
correction involving $V_f$; assembling these corrections produces the effective variance
$\sigma^2_{\mathrm{eff}}$.  The test function $e^{\theta\epsilon q}$ remains unchanged---it is the
\emph{analysis tool} that changes, making the Poisson equation a modular component that
can be combined with the transform method wherever Markov modulation introduces
dependence.
\end{DTonly}

%% ----- Model Description -----
\subsubsection*{Model Description for the Markov-Modulated $M/M/1$ queue}
Consider a single-server queue operating in continuous time, in a random environment denoted by $Z(t)$.
Let $\{Z(t)\}_{t \ge 0}$ be an ergodic continuous-time Markov chain (CTMC)
on a finite state space $\mathcal{Z}$ with transition rate matrix
$Q = [\alpha_{ii'}]_{i,i'\in\mathcal{Z}}$ and stationary distribution $\pi$.
We use $\alpha_{i\bullet} := \sum_{i' \neq i} \alpha_{ii'}$ to denote the total
rate out of state $i$.  Conditioned on the environment $Z(t) = i$, arrivals occur at rate $\lambda_i$ and
the server processes at rate $\mu_i$.
The long-run average rates are $\lambda=\mathbb{E}_\pi[\lambda_Z]$ and $\mu=\mathbb{E}_\pi[\mu_Z]$,
with $\epsilon := \mu - \lambda > 0$ governing the proximity to heavy traffic.
The state of the system at time $t$ is the pair $(Z(t), q(t))$, which forms a
positive-recurrent CTMC for all $\epsilon > 0$.  We write $q$ for a
random variable that follows the stationary distribution of $q(t)$.
This model is studied in \citet{HL-Gro-2026-Markov-Modulated} under a countable state space $\mathcal{Z}$. For ease of exposition, we focus on the finite-state-space model here.
%% ----- Why standard i.i.d.\ analysis fails -----
\subsubsection*{Why Standard Analysis Fails}
An essential step to compute the \emph{transform equation} \eqref{eq:ct-transform-equation} is the factorization
$\E{e^{\theta\epsilon q}\bigl(\mu\,e^{-\theta\epsilon} - \lambda\bigr)} = \bigl(\mu\,e^{-\theta\epsilon} - \lambda\bigr)\E{e^{\theta\epsilon q}}$, which is possible because $\lambda$ and $\mu$ are constants. With Markov modulation, \new{since} the rates $\lambda$ and $\mu$ depend on $Z(t)$, which in turn
is correlated with $q(t)$ through the joint chain $(Z(t), q(t))$, setting the
generator drift to zero produces cross-terms $\E{e^{\theta\epsilon q}\bigl(\mu_Z\,e^{-\theta\epsilon} - \lambda_Z\bigr)}$ (which is of the form $\mathbb{E}\!\left[e^{\theta \epsilon q}\,f(Z)\right]$)
that cannot be factored.
%% ----- The Poisson Equation -----
\subsubsection*{The Poisson Equation}
The Poisson equation of a function $f$ asks for a function $V_f$ such that \begin{equation}\label{eq:poisson-ct}
  V_f(i)
  \;=\;
  \frac{f(i) - \bar{f}}{\alpha_{i\bullet}}
  \;+\; V_f^+(i),
  \qquad \forall\, i \in \mathcal{Z},
\end{equation}
where $\bar{f}=\E{f(Z)}$ and $V_f^+(i) := \sum_{i' \neq i} \tfrac{\alpha_{ii'}}{\alpha_{i\bullet}}\,V_f(i')$
is the expected value of $V_f$ one transition after state $i$.
\new{Using matrix vector notation, this equation is equivalent to} $QV_f = \bar{f} - f$. Noting that this equation is similar to the Bellman equation for the policy evaluation problem, intuitively, the solution $V_f(i)$ characterizes the total expected future deviation of $f(Z)$ from its mean, starting from state $i$.
The central role of the Poisson equation is to bound the difference between the cross-terms $\mathbb{E}\!\left[e^{\theta \epsilon q}\,f(Z)\right]$ and their independent counterparts $\E{e^{\theta\epsilon q}} \, \E{f(Z)}$, as follows:
\begin{equation}\label{eq:poisson-lemma-ct}
  \mathbb{E}\!\left[e^{\theta \epsilon q}\, f(Z)\right]
  \;=\;
  \mathbb{E}\!\left[e^{\theta \epsilon q}\right]\bar{f}
  \;+\;
  (e^{\theta\epsilon}-1)\,
  \mathbb{E}\!\left[e^{\theta\epsilon q}\, V_f(Z)\,
    (\lambda_Z - \mu_Z)\right]
  \;+\;
  O\!\left(\epsilon^{2}\right).
\end{equation}
The first term on the right-hand side is the decoupled approximation (as if $f(Z)$ and $q$
were independent), and the second is a correction of order $\epsilon$ that
captures the covariance between the environment and the queue through the
Poisson solution $V_f$. Equation \eqref{eq:poisson-lemma-ct} is formally proved by \citet[Theorem 2]{HL-Gro-2026-Markov-Modulated}. Now we are ready to use the transform method.
\subsubsection*{Step 1: Capturing the System Dynamics}
To avoid any existence issues, here we work with the one-sided Laplace transform, i.e., $e^{\theta\epsilon q}$ with $\theta < 0$.
We follow exactly the same steps as in the $M/M/1$ queue in Section~\ref{sec:ct-transform-method}, keeping in mind that the arrival and service rates depend on $Z$. Specifically, we obtain
\begin{equation}
  \mathbb{E}\!\left[e^{\theta\epsilon q}
    \bigl(\mu_Z\,e^{-\theta\epsilon} - \lambda_Z\bigr)\right]
  \;=\;
  e^{-\theta\epsilon}\,\mathbb{E}\!\left[\mu_Z\,
    \mathbbm{1}_{\{ q=0\}}\right]. \label{eq:intermediate_mm1_drift}
\end{equation}
\new{Similar} to the $M/M/1$ queue, we set $\theta = 0$ in \eqref{eq:intermediate_mm1_drift} to obtain $\E{\mu_Z\,\mathbbm{1}_{\{q=0\}}} = \E{\mu_Z - \lambda_Z} = \epsilon$. Hence, we obtain the \emph{transform equation}
\begin{equation}\label{eq:mm-transform-ct}
  \mathbb{E}\!\left[e^{\theta\epsilon q}
    \bigl(\mu_Z\,e^{-\theta\epsilon} - \lambda_Z\bigr)\right]
  \;=\;
  \epsilon\, e^{-\theta\epsilon}.
\end{equation}
\subsubsection*{Step 2: Second-Order Approximation}
To compute the left-hand side of the \emph{transform equation}, we start by taking the Taylor expansion of $e^{-\theta \epsilon}$ to first order:
\begin{align}
  \mathbb{E}\!\left[e^{\theta\epsilon q}
    \bigl(\mu_Z\,e^{-\theta\epsilon} - \lambda_Z\bigr)\right]
  &=\underbrace{\mathbb{E}\!\left[e^{\theta \epsilon q}\,(\mu_Z - \lambda_Z)\right]}_{(\star)}
  \;-\; \theta \epsilon\,\underbrace{\mathbb{E}\!\left[e^{\theta \epsilon q}\,\mu_Z\right]}_{(\star\star)}
  \;+\; O(\epsilon^2). \label{eq:mm-expanded-ct}
\end{align}
The two expectations $(\star)$ and $(\star \star)$ involve cross-terms between
$e^{\theta \epsilon q}$ and functions of the environment state $Z$, which is precisely
where the Poisson equation is needed.
\subsubsection*{Step 3: Simplifying via the Poisson Equation}
Define $h(Z) := \mu_Z - \lambda_Z$, so $\bar{h} = \epsilon$.
Apply \eqref{eq:poisson-lemma-ct} with $f = h$:
\begin{equation}\label{eq:star-ct}
  \mathbb{E}\!\left[e^{\theta \epsilon q}\,(\mu_Z - \lambda_Z)\right]
  \;=\;
  \epsilon\,\mathbb{E}\!\left[e^{\theta\epsilon q}\right]
  \;+\; (e^{\theta \epsilon}-1)\mathbb{E}\!\left[e^{\theta \epsilon q}\,V_h(Z)\,(\lambda_Z-\mu_Z)\right]
  \;+\; O\!\left(\epsilon^2\right).
\end{equation}
Now define
$\kappa(Z):= V_h(Z)(\mu_Z-\lambda_Z)$
and apply \eqref{eq:poisson-lemma-ct} a \emph{second time} with $f = \kappa$.
Similarly, we apply \eqref{eq:poisson-lemma-ct} with $f=\mu_Z$ to bound $(\star\star)$. Putting everything together, and simplifying terms, we obtain
\begin{equation}\label{eq:mm-mgf-prelimit-ct}
  \mathbb{E}\!\left[e^{\theta \epsilon q}\right]
  \;=\;
  \frac{1}{1 - \theta \left(\mu+\mathbb{E}\!\left[\kappa(Z)\right]\right)} + O\left(\epsilon\right).
\end{equation}
Taking $\epsilon \to 0$ in \eqref{eq:mm-mgf-prelimit-ct}, we recognize the
right-hand side as the Laplace transform of an exponential random variable:
\begin{equation}\label{eq:mm-limit-ct}
  \epsilon q
  \;\xrightarrow{d}\;
  \mathrm{Expo}\!\left(\frac{1}{\mu+\kappa^*}\right),
\end{equation}
where $q$ denotes the steady-state queue length and, assuming the limit exists, we define
\begin{equation}\label{eq:kstar-ct}
  \kappa^* \;:=\; \lim_{\epsilon\to 0}\,\mathbb{E}_\pi[\kappa(Z)]
  \;=\; \lim_{\epsilon\to 0}\,\mathbb{E}_\pi\!\left[V_h(Z)\,(\mu_Z-\lambda_Z)\right].
\end{equation}
This has the same \emph{form} as the classical result
$\epsilon  q \stackrel{d}{\to} \mathrm{Expo}(1/\mu)$. The only difference is that the mean $\mu$ is replaced by $\mu + \kappa^*$, the effective variance, which encodes the
temporal correlations.
\paragraph{Beyond fast-mixing assumptions.}
A striking feature of \eqref{eq:mm-limit-ct} is that its validity requires only
finiteness of certain moments of the Poisson solutions $V_h$ and $V_\kappa$, which are trivial for finite-state Markov chains.
These conditions depend on the modulating chain through its Poisson solutions $V_h$ and $V_\kappa$, but do not require an explicit a priori bound on the mixing time.
Classical approaches to Markov-modulated queues typically impose an explicit fast-mixing condition relative to the heavy-traffic timescale. The Poisson equation instead absorbs the temporal correlations into the constant $\kappa^*$, so the chain's mixing behavior enters the limiting distribution through this constant rather than as a separate assumption on the timescale of convergence.

\subsection{Matching Queue with State-Dependent Arrivals}\label{sec:matching-queue}
We now turn to a model that arises naturally in two-sided marketplaces \new{such as} ride-hailing platforms (\citet{varma2023dynamic}\new{)} and freelance labor markets, where two types of agents (e.g., customers and servers) arrive and are matched upon meeting. Two-sided queues also naturally show up in many other applications, such as payment channel networks (\citet{sushil_blockchain}), assemble-to-order systems (\citet{matchingqueues}), quantum switch (\citet{zubeldia2026matching, bhambay2025optimal}), organ donation (\citet{kerimov2025optimality, comte2021online}), network revenue management (\citet{gupta2024greedy}), etc. We also refer the reader to \citet{caldentey2009fcfs, adan2012exact} for further references. Here we consider \new{a single two-sided queue, which models a taxi-stand, and forms the basic building block of stochastic matching networks.}
\subsubsection*{What's new here (Preview).} The arrival rate now depends on the
queue state through a pricing control, which couples the transform
to the state and produces an \emph{implicit} functional equation
rather than a closed-form MGF. The new tool is the
\emph{inverse Fourier transform}, which converts the implicit
frequency-domain equation into an explicit ODE for the density in
the spatial domain.

\subsubsection*{Model Description}
Consider a discrete-time matching queue with two sides: a \emph{customer} side queue of length $q^c$ and a \emph{server} side queue of length $q^s$.  Whenever both sides are nonempty, agents are matched immediately and leave the system.  As a consequence, at most one side can be nonempty at any time\new{,} similar to how a non-idling server in a classical queue ensures that work and idleness do not coexist. The state of the system is completely described by the \emph{imbalance}
\[
z \;:=\; q^c - q^s,
\]
which takes values in $\mathbb{Z}$\new{, the set of integers.} Positive values indicate a surplus of customers, and negative values indicate a surplus of servers.  The original queue lengths are recovered via $(q^c, q^s) = (z, 0)$ if $z > 0$ and $(q^c, q^s) = (0, -z)$ if $z < 0$.

Let us first consider the setting where, in each time slot,  exogenous arrivals to each side are i.i.d.\ random variables with mean $\lambda$ and $\mu$ respectively. Unlike the classical single-server queue, the matching queue with larger server arrival rate $(\lambda < \mu)$ is transient as the server queue escapes to infinity. Similarly, $\lambda > \mu$ also leads to an unstable system. Even with equal arrival rates $(\lambda = \mu)$, the system is \emph{null recurrent}\new{, since} the imbalance performs a symmetric random walk and does not have a stationary distribution. External intervention is therefore needed for stability \new{to ensure that there is a stationary distribution. One} natural mechanism is \emph{dynamic pricing}\new{, where the} platform adjusts prices based on the current imbalance, which in turn modulates the effective arrival rates to each side.

\subsubsection*{State-Dependent Arrivals via Pricing}
We model the effect of pricing through bounded control functions $\phi^c(\cdot)$ and $\phi^s(\cdot)$ that modify the arrival rates as functions of the imbalance.  Specifically, the effective arrival rates to the customer and server sides are
\[
\lambda(z) = \lambda^\star + \epsilon\,\phi^c\!\left(\frac{z}{\tau}\right), \qquad \mu(z) = \lambda^\star + \epsilon\,\phi^s\!\left(\frac{z}{\tau}\right),
\]
where $\epsilon > 0$ is the magnitude of the control and $\tau > 0$ is a spatial scale parameter (a threshold). \new{When the imbalance} $z$ is large and positive (many customers waiting), the platform might raise customer-side prices to discourage more customer arrivals ($\phi^c$ increases) and/or lower server-side prices to attract server arrivals ($\phi^s$ decreases), and vice versa.
For stability, the control must provide a \new{drift towards the origin. There exist} $\delta > 0$ and $K > 0$ such that $\phi^c(x) - \phi^s(x) < -\delta$ for all $x > K$ and $\phi^c(x) - \phi^s(x) > \delta$ for all $x < -K$.
One such example of state-dependent control is the two-price policy, corresponding to $\phi^c(x)=-\mathbbm{1}\{x>1\}$ and $\phi^s(x)=-\mathbbm{1}\{x<-1\}$. Specifically, we have
\begin{align}
\label{eq: two_price_policy}
\lambda(z)=\lambda^\star-\epsilon\mathbbm{1}\{z>\tau\}, \ \mu(z)=\lambda^\star-\epsilon \mathbbm{1}\{z<-\tau\} \quad \forall z \in \mathbb{Z}.
\end{align}
The customer and server arrivals to the system are given by independent random variables $a^c(z)$ and $a^s(z)$ when the imbalance is $z$, with mean $\mathbb{E}[a^c(z)] = \lambda(z)$, $\mathbb{E}[a^s(z)] = \mu(z)$ and variance $\text{Var}[a^c(z)] = \sigma^c(\lambda(z))$, $\text{Var}[a^s(z)] = \sigma^s(\mu(z))$. We assume $a^c, a^s$ have finite support and the variances $\sigma^c(\lambda(z))$ and $\sigma^s(\mu(z))$ are bounded. Now, the imbalance evolves according to
\begin{equation}\label{eq:imbalance-evolution}
z(k+1) \;=\; z(k) + a^c\bigl(z(k)\bigr) - a^s\bigl(z(k)\bigr).
\end{equation}
Note the key structural difference from the single-server queue \new{is that there} is no positive-part operator and no complementarity condition.  Instead, the state space is \emph{all of $\mathbb{Z}$}, and the arrivals themselves are functions of the state.  This state-dependence is what makes the analysis fundamentally different from the classical case.

\subsubsection*{Heavy-Traffic Regimes}
Heavy traffic is achieved by letting the control vanish, which pushes the system toward null recurrence, similar to the heavy-traffic regime in a classical queue. There are two natural ways to do this\new{. Either} let the control magnitude $\epsilon \to 0$, so the pricing intervention becomes negligible, or let the threshold $\tau \to \infty$, so\new{,} the control is applied only at increasingly extreme imbalance levels. The interplay between $\epsilon$ and $\tau$ gives rise to three regimes, distinguished by how the threshold $\tau$ scales relative to $1/\epsilon$. As summarized in Table~\ref{tab: phase_transition}, we observe a phase transition in the limiting distribution of imbalance from a Laplace distribution when $\tau = o(1/\epsilon)$ to a Uniform distribution when $\tau = \omega(1/\epsilon)$. The in-between regime is when $\tau = \Theta(1/\epsilon)$, where we observe the limiting distribution to be Gibbs as a function of the control curves $(\phi^c, \phi^s)$.
The Gibbs density of the hybrid regime degenerates to the Laplace distribution as one moves toward $\tau = o(1/\epsilon)$, and to the uniform distribution as one moves toward $\tau = \omega(1/\epsilon)$.

\textbf{Cost of control and applications.}  \new{State-dependent pricing is usually implemented through dynamic pricing leading to some cost in terms of lost revenue. The three regimes correspond to different operating points on the underlying trade-off between cost-of-control and system performance} (\citet{varma2022twosidedqueues}). A \emph{low} cost of control favors aggressive intervention, which yields $\tau = o(1/\epsilon)$ \new{leading to} the Laplace regime\new{. This} case captures ride-hailing platforms, where a small perturbation around a static base price is known to be near-optimal (\citet{varma2023dynamic}). A \emph{moderate} cost of control favors a balanced threshold $\tau = \Theta(1/\epsilon)$ and the hybrid (Gibbs) regime\new{. Payment} channel networks such as the Lightning network for Bitcoin fit this case, with blockchain settlement serving as the costly control lever (\citet{sushil_blockchain}). A \emph{high} cost of control discourages frequent intervention and yields $\tau = \omega(1/\epsilon)$ \new{leading to} the uniform regime\new{. Ride-hailing} with significant customer dissatisfaction from surge pricing, or manufacturing systems with substantial machine setup costs, fall under this case. Thus, the same matching-queue model, equipped with the same transform-method analysis, captures three qualitatively distinct operating modes that platform designers encounter in practice.

\begin{table}[ht]
\TABLE{A Summary of Heavy-Traffic Phase Transitions in Matching Queues.
    \label{tab: phase_transition}}
    {
    \begin{tabular}{>{\centering\arraybackslash}p{85pt}| >{\centering\arraybackslash}p{134pt}|>{\centering\arraybackslash}p{127pt}|>{\centering\arraybackslash}p{115pt}}
    \hline
         &  $\tau = o(1/\epsilon)$, e.g., $\tau=1$ &   $\tau = \Theta(1/\epsilon)$, e.g., $\tau = 1/\epsilon$ & $\tau = \omega(1/\epsilon)$, e.g., $\tau = 1/\epsilon^2$ \\
         \hline
         & & & \\
         Bernoulli Arrivals, Two Price Policy & Laplace & Hybrid & Uniform \\ [-13pt]
     & \includegraphics[width=0.2\textwidth]{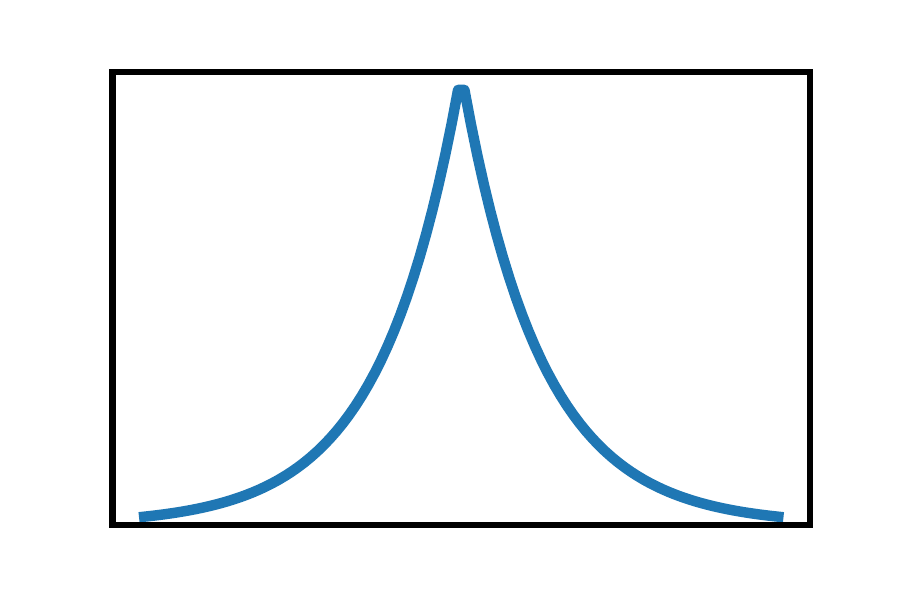} &\includegraphics[width=0.2\textwidth]{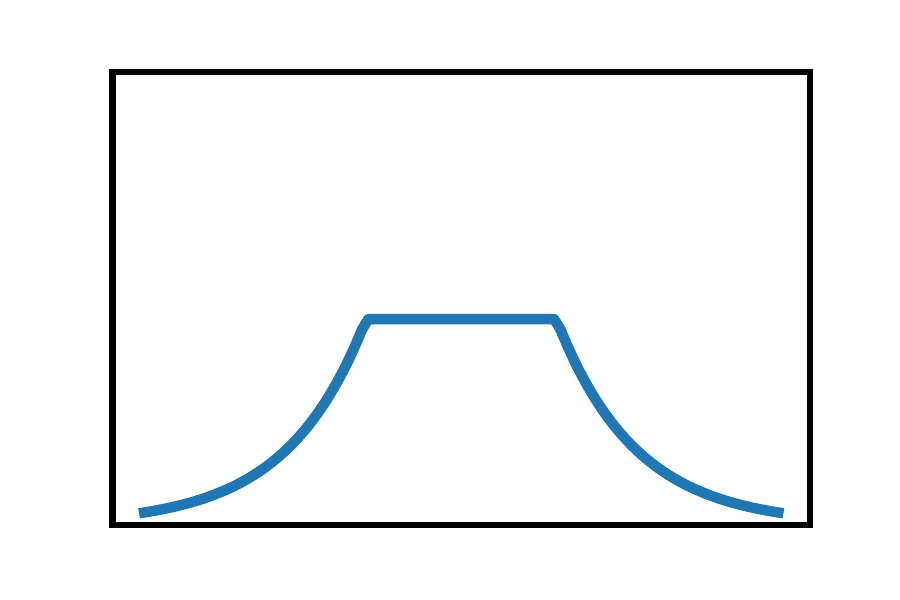} & \includegraphics[width=0.2\textwidth]{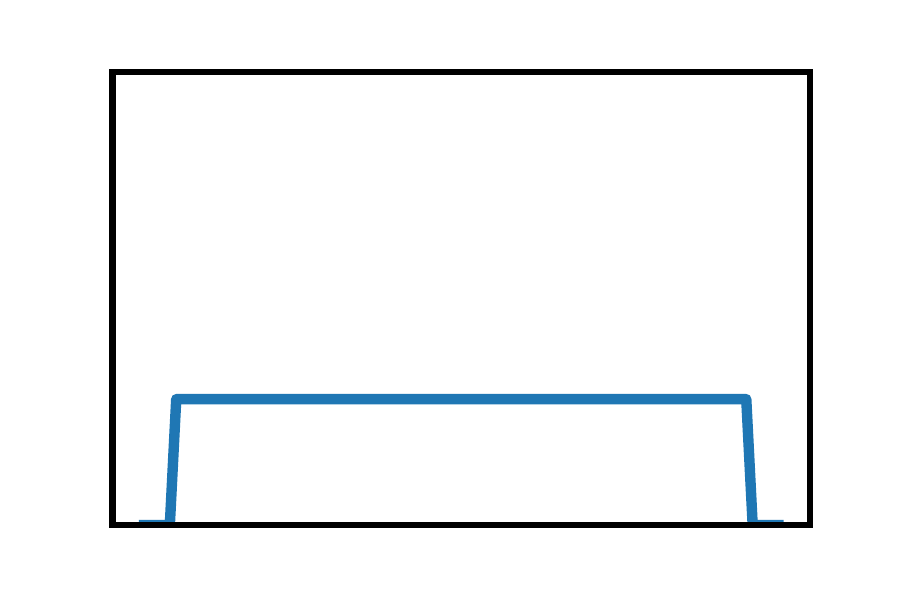} \\ \hline
    General Arrival and Pricing & Asymmetric Laplace & Gibbs Distribution &  Uniform    \\
  \hline
  Proof Technique& (i) Transform method on $|z|$ (ii) Exploiting symmetry around zero & (i) Transform method on $z$ (ii) Inverse transforms to solve functional equation & (i) Transform method within the thresholds (ii) Large $\epsilon$ ensures bounded support \\ \hline
    \end{tabular}}{We refer the reader to \citet{varma2022twosidedqueues} for formal statement of these results.}
\end{table}
\subsubsection*{Hybrid Regime.} \new{Here, we focus on} the hybrid regime and set $\tau = 1/\epsilon$.
\subsubsection*{Step 1: Capturing the System Dynamics}
Since the imbalance $z$ is signed and the state space is all of $\mathbb{Z}$, we work with the \emph{characteristic function} $\mathbb{E}[e^{j\omega\epsilon z}]$ for $\omega\in\mathbb{R}$ rather than the MGF\new{.} The CF always exists, and as we will see in Step 3, it plays naturally with the inverse Fourier transform. \new{Now, setting the drift of $z(k)$ to zero in steady state, we get,}
\[
\mathbb{E}\bigl[e^{j\omega\epsilon\, z(k+1)}\bigr] \;=\; \mathbb{E}\bigl[e^{j\omega\epsilon\, z(k)}\bigr] \overset{\eqref{eq:imbalance-evolution}}{\implies} \mathbb{E}\Bigl[e^{j\omega\epsilon\, z}\,\Bigl(e^{j\epsilon\omega\,(a^c(z) - a^s(z))} - 1\Bigr)\Bigr] \;=\; 0,
\]
where $z$ denotes the steady-state imbalance.
Note that, unlike the single-server queue, there is no complementarity condition to \new{exploit. Instead, we have the} state-dependent arrivals inside the expectation.

\subsubsection*{Step 2: Second-Order Approximation}
The next step is to expand the inner exponential for small $\epsilon$. Using a Taylor expansion of $e^{j\epsilon\omega(a^c(z) - a^s(z))} - 1$ and carefully tracking the mean and variance of $a^c(z) - a^s(z)$ (which are themselves functions of $z$ through the control), \new{and taking the limit,} one arrives at the \emph{implicit} functional equation
\begin{equation}\label{eq:implicit-cf}
j\omega\,\mathbb{E}\!\left[e^{j\omega\xi}\right] \;=\;\mathbb{E}\!\left[e^{j\omega\xi}\, g(\xi)\right] \quad \text{with} \quad g(x) \;=\; \frac{2}{\sigma^c(\lambda^\star) + \sigma^s(\lambda^\star)} \left(\phi^s(x) - \phi^c(x)\right),
\end{equation}
where $\xi$ denotes the scaled limiting imbalance ($\epsilon z$ in the limit) and the function $g$ captures the impact of the control curves $(\phi^c, \phi^s)$. \new{To obtain this equation, one has to first establish tightness of the scaled imbalances, so that the limit is meaningful (along a subsequence). This is separately done using standard Lyapunov drift argument.}

\subsubsection*{Step 3: Solving via Inverse Fourier Transform}
\new{In the single-server queue, at this point, one obtains Equation~\eqref{eq:limiting-mgf}, which gives the limiting MGF explicitly. However, here, we have an implicit equation on the characteristic function, which needs to be solved. We do this by going from the characteristic function to the density using the inverse Fourier transform.} Suppose the distribution of $\xi$ has a density $f(\cdot)$ and assume that it is continuously differentiable.  Then equation~\eqref{eq:implicit-cf} can be written as
\[
j\omega \int_{-\infty}^{\infty} e^{j\omega x}\, f(x)\, dx  \;=\; \int_{-\infty}^{\infty} e^{j\omega x}\, g(x)\, f(x)\, dx.
\]
The right-hand side is the Fourier transform of the function $g(x) f(x)$\new{. The} left-hand side is $j\omega$ times the Fourier transform of $f(x)$\new{, which is the Fourier transform of $-f'(x)$ due to the differentiation property of the Fourier transform. Taking the inverse Fourier transform, we get,}
\begin{equation}\label{eq:ode-density}
g(x)\, f(x) \;=\; -f'(x).
\end{equation}
This is a first-order linear ODE for the density $f$, which can be solved by dividing both sides by $f(x)$ and integrating, giving the solution
\begin{equation}\label{eq:gibbs-density}
f(x) \;=\; c\, \exp\!\left(-\int_0^x g(y)\, dy\right),
\end{equation}
where $c > 0$ is the normalizing constant that ensures $\int_{-\infty}^{\infty} f(x)\, dx = 1$.  The density~\eqref{eq:gibbs-density} is a \emph{Gibbs distribution} with energy function $G(x) = \int_0^x g(y)\, dy$.  Its shape is entirely determined by the control functions $\phi^c$ and $\phi^s$ through $g$. \new{Here, using the characteristic function enabled us to exploit the properties of the Fourier transform to obtain the simple ODE.}

\textit{Formalizing the inverse transform technique:}
The derivation above assumed that the limiting distribution admits a density, which need not hold a priori. The argument can nonetheless be made rigorous using standard distribution-theoretic techniques, under mild smoothness assumptions on the control functions $\phi^c$ and $\phi^s$; we refer the reader to~\citet{varma2022twosidedqueues} for the full details.
\paragraph{Contrast with the abandonment setting.}  It is instructive to compare the differential equation here with the ODE~\eqref{eq:mgf-ode} that arose in the single-server queue with abandonment (see Section \ref{sec:abandonment}). In the abandonment setting, the abandonment term introduces a derivative $dM/d\theta$ in the \emph{transform domain},
yielding a first-order ODE for the \emph{MGF} $M(\theta)$ itself.  Here, the implicit functional equation is in the \emph{transform domain}, and we used the inverse Fourier transform to obtain a first-order ODE for the \emph{density} $f(x)$\new{.} In both cases, the transform method encounters an implicit or differential equation rather than an explicit expression.

\subsubsection*{The Laplace regime}
To contrast and for completeness, we briefly sketch the Laplace regime, where we fix $\tau = 1$. Here the threshold $(\tau = 1)$ is small relative to $1/\epsilon$, so the control acts almost everywhere on the state space. The main idea here is that $|z|$ behaves like the queue length of a single-server queue (assuming symmetric control curves $\phi^c, \phi^s$ around the origin) in heavy traffic\new{, and so,} the natural object of study is $|z|$ rather than $z$ itself. The methodological idea is to engineer two complex exponential test functions. The first, $e^{j\omega\epsilon|z|}$, yields the characteristic function of $\epsilon|z|$ in closed form and reveals that $\epsilon|z|$ converges to an exponential random variable\new{,} exactly as in the single-server queue. The second, $\mathrm{sign}(z)\,e^{j\omega\epsilon|z|}$, establishes that the limiting distribution of $\epsilon z$ is symmetric around zero. Together, these two facts yield that $\epsilon z$ converges to a Laplace distribution. Methodologically, the Laplace regime extends the transform method differently from the hybrid regime\new{.} Instead of inverting an implicit Fourier-domain equation, one engineers multiple test functions that exploit symmetry to pin down the limit. We refer the reader to \citet{varma2022twosidedqueues} for the full derivation.

\section{Stochastic Networks}\label{sec:networks}
We now turn to \emph{stochastic networks} of multiple
interacting queues, where the state is inherently multi-dimensional. \new{The
same three-step recipe applies, but each step becomes harder because the
queue-length vector is multi-dimensional and as such Step~3 now may yield a
\textit{multi-dimensional functional equation}. The key simplification is
\emph{state space collapse} (SSC).} In heavy
traffic the routing or scheduling policy forces the queue-length vector
$\bm q=(q_1,\dots,q_n)$ to concentrate near a smaller set, and the
\new{\textit{geometry}} of that set governs how hard Step~3 is (Figure~\ref{fig:ssc}).
We illustrate this on two networks of increasing difficulty.
Under load balancing with the Join-the-Shortest-Queue (JSQ) policy, all queues
equalize, so $\bm q$ collapses onto a single ray, the line $q_1=\dots=q_n$
(Figure~\ref{fig:ssc}, left panel); the problem becomes effectively one-dimensional and
Step~3 reduces to a scalar equation. In the input-queued switch under
MaxWeight scheduling, $\bm q$ still collapses, but only onto a lower-dimensional
\emph{cone} $\mathcal{K}$ that is itself multi-dimensional
(Figure~\ref{fig:ssc}, right panel); Step~3 then becomes a multi-dimensional
functional equation, which is far harder to solve.
 
\definecolor{qbrown}{RGB}{150,86,20}
\definecolor{setfill}{RGB}{183,206,219}
\definecolor{diagline}{RGB}{56,110,135}
\tikzset{
    ax/.style={-{Stealth[length=3mm]},black,thick},
    qvec/.style={qbrown,line width=1.3pt,-{Stealth[length=3mm]}},
    gap/.style={qbrown,line width=0.7pt,dashed},
    setline/.style={diagline,line width=1.6pt}}
\begin{figure}[ht]
\centering
% ---------- Panel (a): CRP / line (2D) ----------
\begin{minipage}{0.48\textwidth}
  \centering
  \begin{tikzpicture}[scale=0.62]
    \draw[ax] (0,0) -- (5.6,0) node[black,right,inner sep=2pt] {$q_1$};
    \draw[ax] (0,0) -- (0,5.6) node[black,above,inner sep=2pt] {$q_2$};
    \draw[setline] (0,0) -- (5.0,5.0);
    \node[diagline,anchor=south,inner sep=2pt]
      at (5,5) {\footnotesize $q_1\!=\!q_2$};
    \coordinate (Pa) at (3.25,3.25);
    \coordinate (Qa) at (2.40,4.10);
    \draw[gap] (Qa) -- (Pa);
    \fill[diagline] (Pa) circle (1.6pt);
    \draw[qvec] (0,0) -- (Qa);
    \node[qbrown] at (1.05,2.35) {$\bm q$};
  \end{tikzpicture}
\end{minipage}
\hfill
% ---------- Panel (b): non-CRP / cone in a plane (3D) ----------
\begin{minipage}{0.48\textwidth}
  \centering
  \begin{tikzpicture}[scale=0.62,
    x={(0.95cm,-0.05cm)},
    y={(-0.75cm,-0.5cm)},
    z={(0cm,1cm)}]
    % Axes
    \draw[ax] (0,0,0) -- (7,0,0);
    \node[black,anchor=west,inner sep=2pt] at (7,0,0) {$q_1$};
    \draw[ax] (0,0,0) -- (0,3.0,0);
    \node[black,anchor=north east,inner sep=2pt] at (0,3.0,0) {$q_2$};
    \draw[ax] (0,0,0) -- (0,0,4);
    \node[black,anchor=south,inner sep=2pt] at (0,0,4) {$q_3$};
    \coordinate (Lb) at (8.0,1.0,3.0);
    \coordinate (Ub) at (12.0,9.0,3.0);
    \fill[setfill,opacity=0.85]
      (0,0,0) -- (Lb)
      .. controls (10.0,3.5,3.3) and (12.0,7.5,3.3) ..
      (Ub) -- cycle;
    \draw[diagline,line width=1.4pt] (0,0,0) -- (Lb);
    \draw[diagline,line width=1.4pt] (0,0,0) -- (Ub);
    \coordinate (Pb) at (4.4,1.3,1.5);
    \coordinate (Qb) at (3.8,1.6,3.0);
    \draw[qvec] (0,0,0) -- (Qb);
    \node[qbrown] at (2.0,0.9,2) {$\bm q$};
    \draw[gap] (Qb) -- (Pb);
    \fill[diagline] (Pb) circle (1.6pt);
    \node[diagline] at (7.0,4.0,2.0) {\large$\mathcal{K}$};
  \end{tikzpicture}
\end{minipage}
\caption{Illustration of state space collapse in heavy traffic. The
queue-length vector $\bm q$ concentrates near a low-dimensional set; the
dashed segment marks the small residual distance to that set. The left panel shows SSC to a single dimension (CRP), and the right panel to a multi-dimensional subspace.}
\label{fig:ssc}
\end{figure}
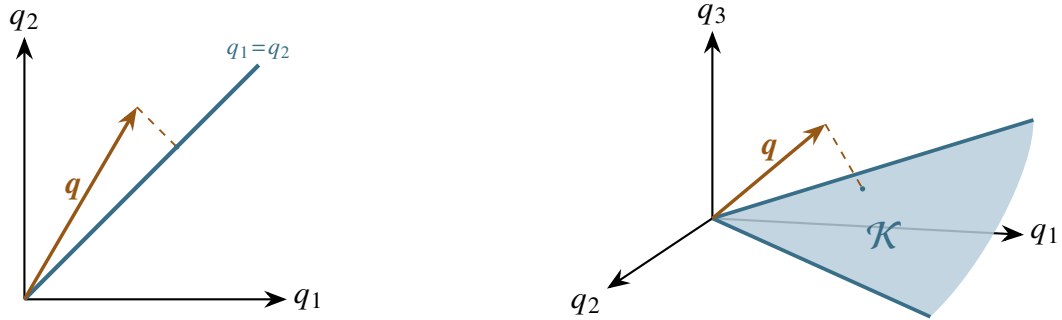
 
\subsection{Load Balancing Systems}
Load balancing is a fundamental problem in the design and operation of large-scale service systems. In cloud data centers, millions of incoming requests must be dispatched across thousands of servers in real time\new{. The} routing policy directly affects the delay and the overall resource utilization of the system. Similar load-balancing challenges arise in web server farms, content delivery networks, call center staffing, and hospital patient routing.
Effective routing in these domains must account for stochastic demand and resource constraints to prevent systemic bottlenecks (\citet{green2006health, armony2015patient}). \new{A simple and well-studied load-balancing policy is \emph{Join-the-Shortest-Queue} (JSQ), where each arriving job is routed to the server with the fewest waiting jobs.} JSQ has been known to be delay-optimal (in a stochastic ordering sense) among all non-anticipating policies  as shown by \citet{Menich_OptimalityJSQ_majorization, winston_JSQ_1977}. Here, we present the limiting distribution and the tail bound of a load balancing system \new{in the} classic heavy traffic regime using the transform method (\citet{jhunjhun2024exponential}) and briefly discuss \new{extensions to} many-server heavy traffic, abandonment and Markov modulation for a load balancing system.
 
\subsubsection{Load Balancing under Classic Heavy Traffic}\label{sec:jsq-corrected}
Consider a system of $n$ identical parallel servers, operating in continuous time.  Customers arrive according to a Poisson process with rate $\lambda = n(1-\epsilon)$, and each server processes jobs independently at rate~$1$ with exponentially distributed service times. Here, $\epsilon$ denotes the heavy-traffic parameter. \new{In the classical heavy traffic regime,} the system size $n$ is fixed and $\epsilon \rightarrow 0$. Upon arrival, each customer joins the shortest queue (with ties broken uniformly at random).  Let $q_i(t)$ denote the number of customers at server $i$ (including the one in service), so the system state is $\mathbf{q}(t) = (q_1(t), \ldots, q_n(t))$.
Define the \emph{total queue length} $Q(t) = \sum_{i=1}^n q_i(t)$ and the \emph{scaled total queue length} $\hat{Q} = \epsilon Q$\new{. Under JSQ, in the heavy-traffic limit, we establish that}
\begin{equation} \label{eq:jsq_limiting_result}
n\epsilon(q_1, \hdots, q_n) \overset{d}{\to} (Y, \hdots, Y), \quad \text{where} \quad Y \sim \text{Expo}(1).
\end{equation}
\paragraph{Interpretation.}
\new{This result demonstrates two phenomena.} First,
on the heavy-traffic scale all individual queues collapse to one
another and to the scaled average: each $n\epsilon\,q_i$ has the same exponential limit $Y \sim \mathrm{Expo}(1)$
as $\epsilon Q$. Second, \new{the heavy-traffic limit of $\epsilon Q$ matches that
of} a single $M/M/1$ queue with arrival rate $\lambda = n(1-\epsilon)$
\new{and service rate $n$. Thus, on the heavy-traffic scale, the $n$-server system under JSQ behaves as a single server with service rate $n$, with no performance penalty for distributing the work across $n$ servers. Such a result is possible because the load-balancing system under JSQ satisfies the so-called complete resource pooling (CRP) condition, which ensures that there is a single bottleneck.}
To establish \eqref{eq:jsq_limiting_result}, we proceed in two steps. First, we show that the queue lengths are nearly \new{equal in} heavy traffic, $q_i \approx Q/n$ for all $i$\new{, a} property known as \emph{state space collapse} (SSC). Second, we apply the transform method to the total queue length $Q$ to obtain the heavy-traffic limit.
\subsubsection*{State Space Collapse}
\new{Since we want to show that all the queues do not deviate from their average, it is natural to work with the worst-case deviation, and so, a test function of the form $e^{\max_i|q_i-Q/n|}$, i.e., the exponential of the $\ell_{\infty}$ norm of the deviation vector. However, such a function is not amenable to clean drift analysis because of the max function, and so, we use a smooth approximation of it, viz., $\sum_{i=1}^n \left( e^{\beta(q_i - Q/n)} + e^{-\beta(q_i - Q/n)}\right)$.}\footnote{\citet{jhunjhun2024exponential} used the $\ell^\infty$ norm to get tight SSC bounds, crucial to obtain tail bounds.  However, just to obtain (classic) heavy traffic limit, it is sufficient to work with the $\ell^2$ norm, which is easier to work with but provides weaker SSC guarantees, e.g., see \citet{hurtado2020transform}}. Similar to the first two steps of the transform method, we set the drift of this test function to zero in steady state, use the queue evolution equation, carefully bound the resultant terms, and use the symmetry of the system to obtain $\mathbb{E}\left[e^{\beta|q_i-Q/n|}\right] = O(1)$, which shows that $q_i$ does not deviate too much from $Q/n$ with high probability.
\subsubsection*{Step 1 and 2: Capturing the System Dynamics and SSC Approximation}
Once SSC is established, the system behaves approximately as if the $n$ servers are pooled into a single \new{server} with rate $n$. Using the exponential test function $e^{\theta \hat{Q}}$ in steady state and using the queue evolution equation\new{,} one obtains the following expression for the MGF of the scaled total queue length:
\begin{equation}\label{eq:jsq-mgf-exact}
\mathbb{E}\bigl[e^{\theta \hat{Q}}\bigr] \;=\; \frac{ \sum_{i=1}^n \mathbb{E}\bigl[\mathbbm{1}_{\{q_i = 0\}}\, e^{\theta \hat{Q}}\bigr]}{n - \lambda e^{\epsilon \theta}} \quad \text{ for all } \quad \theta < \theta_\epsilon,
\end{equation}
where $\theta_\epsilon$ is the largest value for which the denominator $n - \lambda e^{\epsilon \theta} > 0$. One can verify that $\theta_\epsilon \to 1$ as $\epsilon \to 0$. Note that the numerator of \eqref{eq:jsq-mgf-exact} involves the idleness indicators $\mathbbm{1}_{\{q_i = 0\}}$ similar to the $M/M/1$ queue setting \eqref{eq:ct-transform-before}.
A key challenge here compared to $M/M/1$ queue is that, when $n=1$, we have $\sum_{i=1}^n \mathbb{E}\bigl[\mathbbm{1}_{\{q_i = 0\}}\, e^{\theta \hat{Q}}\bigr] = \mathbb{E}\bigl[\mathbbm{1}_{\{q_1 = 0\}}\, e^{\epsilon \theta q_1}\bigr] = \mathbb{E}\bigl[\mathbbm{1}_{\{q_1 = 0\}}\bigr]$, yielding a closed form expression as in \eqref{eq:ct-transform-equation}. For $n>1$, we do not have such a simplification in \eqref{eq:jsq-mgf-exact} as terms of the form $\mathbb{E}\bigl[\mathbbm{1}_{\{q_i = 0\}}\, e^{\epsilon \theta q_j}\bigr]$ cannot be directly simplified. Here, SSC allows us to resolve this issue since $\epsilon q_i \approx \epsilon q_j \approx \hat{Q}/n$ for all $i$. More precisely, applying SSC in \eqref{eq:jsq-mgf-exact} gives
\begin{align}\label{eq:jsq-mgf-approx}
\mathbb{E}\bigl[e^{\theta \hat{Q}}\bigr] \;=\; \frac{n \epsilon}{n - \lambda e^{\epsilon \theta}} \cdot \bigl(1 + \delta_{\text{SSC}}(n)\bigr), \ \text{ with } \  |\delta_{\text{SSC}}(n)| \leq \kappa n\epsilon \log \tfrac{1}{\epsilon},
\end{align}
where $\kappa$ is a constant and $\delta_{\text{SSC}}(n)$ quantifies the SSC violation at finite $n$ and $\epsilon$. Here, we use the fact that $\sum_{i=1}^n \mathbb{E}\bigl[\mathbbm{1}_{\{q_i = 0\}}\bigr] = n\epsilon$ simply by substituting $\theta =0$ in  \eqref{eq:jsq-mgf-exact}\new{. The} bound on $\delta_{\text{SSC}}(n)$ in \eqref{eq:jsq-mgf-approx} follows from the bound on the transform of queue imbalance $|q_i - Q/n|$ \new{in} SSC. Details \new{can be found} in \citet{jhunjhun2024exponential}.
\subsubsection*{Step 3: Solving for the MGF} Now we take the limit as $\epsilon \to 0$ implying that the SSC violation $n\epsilon \log(1/\epsilon) \to 0$ as $\epsilon \rightarrow 0$. Thus, we get $\lim_{\epsilon \to 0} \mathbb{E}\bigl[e^{\theta \hat{Q}}\bigr] = 1 / (1-\theta)$, implying that the scaled total queue length converges to $\text{Expo}(1)$ in distribution. Moreover, by SSC, all queues are approximately equal, thus, establishing the joint steady-state distribution in \eqref{eq:jsq_limiting_result}.
\subsubsection*{Tail Bounds}
Similar to Section \ref{sec:tail-bounds}, applying Markov's inequality $\mathbb{P}(\hat{Q} > x) \leq e^{-\theta x}\, \mathbb{E}[e^{\theta \hat{Q}}]$ and optimizing over $\theta \in (0, \theta_\epsilon)$ yields the tail bound
\begin{align}
    \label{eq:jsq-tail-ct}
    \mathbb{P} \left(\hat{Q} > x\right) \;\leq\; \underbrace{2ex}_{\text{pre-exponent}} \;\cdot\; \underbrace{\Big(1 + \kappa\, n\epsilon \log \frac{1}{\epsilon}\Big)}_{\text{SSC violation}} \cdot  \underbrace{e^{-\theta_\epsilon x}}_{\text{pre-limit tail}}, \quad \forall x > 1.
\end{align}
The main difference from the single-server queue setting of Section~\ref{sec:tail-bounds} is that the pre-exponent now additionally has an SSC violation term. As before, note that the \emph{pre-limit tail} $e^{-\theta_\epsilon x}$ gives the exponential decay rate\new{. As} $\epsilon \to 0$, $\theta_\epsilon \to 1$, recovering the $\mathrm{Expo}(1)$ tail of the heavy-traffic limit. A lower bound $\mathbb{P}(\hat{Q} > x) \geq \frac{1}{1-\epsilon} e^{-\theta_\epsilon x}$ (obtained by coupling with a pooled $M/M/1$ queue) shows that the decay rate $\theta_\epsilon$ is \emph{exact}, yielding the large-deviation result $\lim_{x \to \infty} -\frac{1}{x} \log \mathbb{P}(\hat{Q} > x) = \theta_\epsilon$.
 
\subsubsection{Brief on Load Balancing under Many-Server Heavy Traffic}\label{sec:jsq-many-server}
The classic heavy-traffic analysis of Section~\ref{sec:jsq-corrected} fixes the number of servers $n$ and sends $\epsilon \to 0$. In many modern \new{settings,} the number of servers $n$ itself grows, and the heavy-traffic parameter shrinks simultaneously. This is the \emph{many-server heavy-traffic} regime, where one sets $\lambda_n = n(1 - \epsilon_n)$ with $\epsilon_n = n^{-\alpha}$ for some $\alpha > 0$, so that both $n \to \infty$ and $\epsilon_n \to 0$ together.
 
There is a long line of literature on characterizing the queue length \new{behavior} under JSQ (and a related policy called power-of-$d$) in many-server heavy traffic. The mean field regime $(\alpha = 0)$ is analyzed by \citet{dobrushin_po2, mitzenmacher_po2_2}, sub-Halfin-Whitt regime $(\alpha \in (0, 0.5))$ by \citet{liu2018simple}, Halfin-Whitt regime $(\alpha = 0.5)$ by \citet{eschenfeldt_halfin_whitt, braverman_halfin_whitt, banerjee_halfin_whitt, banerjee_halfin_whitt_insensitivity}, super-Halfin-Whitt regime $(\alpha \in (0.5, 1))$ by \citet{super_halfin_whitt_lei, zhao2025many}, non-degenerate slowdown $(\alpha = 1)$ by \citet{GupWal_NDS_JSQ}, super-slowdown regime $(\alpha > 1)$ by \citet{jhunjhun2024exponential}, and classical heavy-traffic $(\alpha = \infty)$ by \citet{atilla, foschini1978basic}. We refer the reader to the survey \citet{load_balancing_survey} for an overview.
% NOTE (proof, p.25): `\citet{guang2025steady}' was struck here.  The margin note asks
% for "G. J. Foschini and J. Salz. A basic dynamic routing problem and diffusion.
% IEEE Trans. Communications, 26(3):320-327, March 1978" to be added next to
% \citet{atilla}; add it to the .bib with key `foschini1978basic'.
 
We now describe how the transform method analysis of Section~\ref{sec:jsq-corrected} extends to the many-server setting. The entire analysis carries over with $\epsilon$ replaced by $\epsilon_n$. The key observation is that the tail bound~\eqref{eq:jsq-tail-ct} and the MGF approximation~\eqref{eq:jsq-mgf-approx} remain valid as long as the SSC violation term $n\epsilon_n \log(1/\epsilon_n) = \alpha n^{1-\alpha} \log n$ vanishes as $n \to \infty$. This holds precisely when $\alpha > 1$, which is the \emph{super-slowdown} regime where $\epsilon_n$ vanishes faster than $1/n$. In this regime, the same limiting result applies, i.e., $\hat{Q} = \epsilon_n Q \xrightarrow{d} \mathrm{Expo}(1)$ and $n\epsilon_n(q_1, \ldots, q_n) \xrightarrow{d} (Y, \ldots, Y)$ with $Y \sim \mathrm{Expo}(1)$, and the pre-limit tail bounds hold with explicit dependence on $n$ and $\alpha$.
 
\subsubsection{Brief on Load Balancing with Customer Abandonment}
Section~\ref{sec:abandonment} showed that for the single-server queue,
\new{abandonment produces a phase transition and the scaled queue length goes} from exponential to
truncated Gaussian to full Gaussian as the system moves from
underloaded to overloaded. \citet{jhunjhunwala2026jsqa} extends this
to the JSQ with abandonment setting (JSQ-A) via a similar SSC result to Section~\ref{sec:jsq-corrected}, with the same three \new{regimes and the corresponding limiting distributions, see}
Table~\ref{tbl: dist}. The three-step recipe remains the same, with the abandonment term again producing a differential equation for the MGF. We refer the reader to \citet{jhunjhunwala2026jsqa} for precise statements.
 
\new{The main difficulty compared to the single server case is that} establishing MGF existence and state space collapse (SSC) both require incorporating the dynamics of abandonment. The SSC for JSQ-A
is \emph{weaker} than for JSQ without \new{abandonment, and only} a second-moment bound $\mathbb{E}[(q_i - Q/n)^2] = O(1)$ is available (instead of an MGF-level bound), since abandonment can induce large jumps that weaken the concentration of $q_i$ around $Q/n$. This weaker SSC nonetheless suffices to carry out the transform-method analysis.
 
Practically, the JSQ-A results have two implications. First, \new{SSC persists even under abandonments, and the system behaves} like a single pooled queue. Second, the phase transition (Table~\ref{tbl: dist}) gives a quantitative framework for capacity planning under \new{abandonment.}
 
\subsubsection{Brief on Load Balancing with Markov-Modulated Arrivals}
The transform method analysis naturally extends to JSQ with Markov-modulated arrivals, by combining state space collapse arguments from Section~\ref{sec:jsq-corrected} and the Poisson equation technique of Section~\ref{sec:markov-modulated}.
Similar to Section~\ref{sec:markov-modulated}, the Poisson equation absorbs the cross-terms between the modulating environment and the queue lengths, revealing an \emph{effective variance} through the constant $\kappa^*$, that replaces the i.i.d.\ variance in the limiting distribution. The result, established by \citet{HL-Gro-2026-Markov-Modulated}, shows that $\epsilon(q_1, \hdots, q_n) \stackrel{d}{\to} (Y, \hdots, Y)$, where $Y$ is an exponential distribution with rate \new{equal to the effective variance} in the classical heavy traffic regime (fixed $n$ and $\epsilon$ goes to $0$).
% NOTE (proof, p.26): the rate "$\mu+\kappa^*$" was struck with no replacement supplied.
% Changed to $1/(\mu+\kappa^*)$ for consistency with \eqref{eq:mm-limit-ct}; please confirm.
 
\subsection{The Input-Queued Switch}\label{sec:switch-corrected}
We now \new{consider an input-queued switch where the complete resource pooling condition is not satisfied.}  This model, analyzed by~\citet{Jhun_heavy_traffic}, is the first application of the transform method to a system with multiple bottlenecks. \new{Here, the collapsed state space is multi-dimensional, introducing new challenges.}
\subsubsection*{Model Description}
Consider a discrete-time $n \times n$ input-queued switch with $n$ input ports and $n$ output ports.  Packets arrive to input port $i$ destined for output port $j$ as an i.i.d.\ process $\{a_{ij}(k)\}_{k=1}^\infty$ with mean rate $\lambda_{ij}$ and variance $\sigma_{ij}^2$.  \new{T}here is a separate queue $q_{ij}$ for each input-output pair $(i,j)$, and the state of the system is the $n^2$-dimensional vector $\mathbf{q} = (q_{ij})$. At each time slot, the switch selects a \emph{schedule} (a perfect matching in the bipartite graph)\new{.} \new{E}ach input port sends to exactly one output port, and each output port receives from exactly one input port. We denote the schedule at time $k$ by $\{s_{ij}(k)\}_{i,j=1}^n$, where $s_{ij} = 1$ means input $i$ sends a packet to output $j$. The scheduling policy is \emph{MaxWeight}: the matching that maximizes $\sum_{i,j} s_{ij} q_{ij}$.  MaxWeight is throughput-\new{optimal, stabilizing} any arrival rate inside the capacity region $\mathcal{C} = \{\lambda : \sum_j \lambda_{ij} < 1\;\forall i,\; \sum_i \lambda_{ij} < 1\;\forall j\}$.
The queue evolution follows the following recursion:
\[
q_{ij}(k+1) \;=\; q_{ij}(k) + a_{ij}(k) - s_{ij}(k) + u_{ij}(k),
\]
where $u_{ij}(k) \geq 0$ is the unused service, with the complementarity condition $\langle \mathbf{q}(k+1), \mathbf{u}(k) \rangle = 0$. We consider the \emph{uniform} heavy-traffic regime: $\lambda_{ij} = (1-\epsilon)/n$ for all $(i,j)$, so that every input and output port carries the same load $1 - \epsilon$.  We further assume the \emph{symmetric variance condition} $\sigma_{ij}^2 = \sigma^2$ for all $(i,j)$.  As $\epsilon \to 0$, \emph{all} ports become simultaneously \new{saturated, meaning that there are multiple bottlenecks,} and the system does \emph{not} exhibit complete resource pooling.
\subsubsection*{State Space Collapse without CRP}
Under uniform heavy traffic, the $n^2$-dimensional queue-length vector collapses, but not to a single dimension.  Define the subspace
\[
\mathcal{S} \;=\; \bigl\{\mathbf{x} \in \mathbb{R}^{n^2} : x_{ij} = w_i + \tilde{w}_j \text{ for some } \mathbf{w}, \tilde{\mathbf{w}} \in \mathbb{R}^n\bigr\},
\]
which is the set of queue-length vectors whose entries decompose as a sum of an input component $w_i$ and an output component $\tilde{w}_j$.  Note that $\mathcal{S}$ has dimension $2n - 1$ (since adding a constant to all $w_i$ and subtracting it from all $\tilde{w}_j$ leaves $\mathbf{q}$ unchanged), implying that there is one degree of freedom. SSC, established by~\citet{MagSri_SSY16_Switch, jhunjhunwala2020lowcomplexity}, asserts that the queue-length vector concentrates near $\mathcal{S}$.
A key difference from the JSQ model is that the input-queued
switch \emph{inherently} has multiple bottlenecks.
Intuitively, for every input row $i$, the total arrival rate
$\sum_{j=1}^n \lambda_{ij} = 1-\epsilon$ saturates the maximum service
rate of $1$ (since at most one packet can leave row $i$ per slot), so
each row acts as a bottleneck; the same holds for every output column.
\new{As such, one would expect state space collapse to a $2n$-dimensional
subspace (as there is one bottleneck per constraint), however, the $n$ row
and $n$ column constraints are not independent (both sets sum to the same
total $n(1-\epsilon)$), which results in a collapse to a $2n-1$ dimensional
subspace, rather than $2n$.}
\subsubsection*{Step 1: Capturing the System Dynamics}
The transform method is applied to the full $n^2$-dimensional queue vector, using the exponential test function $e^{\epsilon \langle \boldsymbol{\theta}, \mathbf{q} \rangle}$ for $\boldsymbol{\theta}$ restricted to the subspace $\mathcal{S}$\new{,} similar to how we picked $\boldsymbol{\theta} = \theta \mathbf{1}$ for JSQ load balancing as the state space collapsed onto a one-dimensional subspace in that case, i.e., all queues are equal\new{.}
Setting the steady-state drift to zero and performing a second-order Taylor expansion in $\epsilon$ yields a pre-limit transform equation.
\subsubsection*{Step 2: Second Order Approximation}
\new{Taking} the limit $\epsilon \to 0$, the resulting functional equation, derived by~\citet{Jhun_heavy_traffic}, takes the form: for all $\boldsymbol{\theta} \in \Theta$,
\begin{equation}\label{eq:switch-functional-corrected}
\bigl(2\langle \boldsymbol{\theta}, \mathbf{1} \rangle - n \sigma^2 \langle \boldsymbol{\theta}, \boldsymbol{\theta} \rangle\bigr)\, L(\boldsymbol{\theta}) \;=\; 2n\, \langle \boldsymbol{\theta}, \mathbf{M}(\boldsymbol{\theta}) \rangle,
\end{equation}
where $L(\boldsymbol{\theta}) = \lim_{\epsilon \to 0} \mathbb{E}[e^{\epsilon \langle \boldsymbol{\theta}, \mathbf{q} \rangle}]$ is the Laplace transform of the limiting distribution, and we also have \textit{boundary terms} represented by $M_k(\boldsymbol{\theta}) = \lim_{\epsilon \to 0} \frac{1}{\epsilon} \mathbb{E}[u_k\, e^{\epsilon \langle \boldsymbol{\theta}, \mathbf{q} \rangle}]$, capturing the transform of the distribution restricted to the boundary where queue $k$ is empty.  The set $\Theta \subset \mathcal{S}$ is the domain on which the Laplace transforms $L(\boldsymbol{\theta})$ and each $M_k(\boldsymbol{\theta})$ are well-defined (bounded and analytic).
The left-hand side of~\eqref{eq:switch-functional-corrected} has the same structure as the single-server transform equation\new{.} \new{T}he term $2\langle \boldsymbol{\theta}, \mathbf{1} \rangle$ comes from the mean drift (analogous to $\epsilon$ in the single-server case), and $n\sigma^2 \langle \boldsymbol{\theta}, \boldsymbol{\theta} \rangle$ comes from the variance of the net input.  The right-hand side involves the boundary terms $\mathbf{M}(\boldsymbol{\theta})$, which play the role of the unused service.  The crucial difference from the CRP case is that there are $n^2$ boundary terms (one per queue), and \new{the functional equation captures the relationship between the distribution of those queues.}
\subsubsection*{Step 3: Solving the Functional Equation}
Solving~\eqref{eq:switch-functional-corrected} requires determining both $L(\boldsymbol{\theta})$ and the $n^2$ boundary functions $M_k(\boldsymbol{\theta})$ simultaneously. To the best of our knowledge, there is no direct method of solving the functional equation in \eqref{eq:switch-functional-corrected}.  Under the symmetric variance condition, \citet{Jhun_heavy_traffic} conjectures the solution to be
\begin{equation}\label{eq:switch-conjecture-corrected}
\epsilon q_{ij} \;\xrightarrow{d}\; \Upsilon_i + \Upsilon_{n+j} - 2\tilde{\Upsilon}, \qquad \forall\, i, j \in \{1, \ldots, n\},
\end{equation}
where $\{\Upsilon_1, \ldots, \Upsilon_{2n}\}$ are independent $\mathrm{Expo}(2/\sigma^2)$ random variables and $\tilde{\Upsilon} = \min_{1 \leq k \leq 2n} \Upsilon_k$.  The limiting queue length for pair $(i,j)$ is thus a \emph{nonlinear} combination of independent exponentials, involving the minimum $\tilde{\Upsilon}$.  An implication is that each marginal $\epsilon q_{ij}$ follows a mixture distribution: with probability $1/n$ it is $\mathrm{Expo}(2/\sigma^2)$ (when the minimum falls on $\Upsilon_i$ or $\Upsilon_{n+j}$), and with probability $1 - 1/n$ it is $\mathrm{Erlang}(2, 2/\sigma^2)$ (the sum of two independent exponentials).
 
\textit{Technical detail on uniqueness:} This conjectured distribution is shown to \emph{satisfy} the functional equation~\eqref{eq:switch-functional-corrected} by explicitly computing the Laplace transform of the random vector and verifying that it solves~\eqref{eq:switch-functional-corrected}.  However, proving that it is the \emph{unique} solution within the class of bounded analytic functions remains open for the general $n \times n$ switch.
To provide justification for the uniqueness conjecture, \citet{Jhun_heavy_traffic} also analyzes the \emph{three-queue system}\new{,} a special case of the switch where the state collapses to a two-dimensional subspace.  In this tractable setting, the functional equation involves two complex variables, and uniqueness is rigorously established \new{by making a connection to the} \emph{Carleman boundary value problem}.
Beyond the discrete-time switch, the transform method also extends to continuous-time systems that do not satisfy the CRP condition.  To demonstrate this, \citet{Jhun_heavy_traffic} analyze the \emph{$N$-system}, a parallel-server model with two customer classes and two server pools operating under MaxWeight with Poisson arrivals and exponential service.
\paragraph{Why CRP versus non-CRP matters.}  In CRP systems, the network has a single bottleneck, and SSC collapses the state to one dimension.  The transform method then produces a one-dimensional functional equation\new{,} either algebraic (solved directly), an ODE (solved by standard techniques), or requires the Poisson equation.  In non-CRP systems, multiple bottlenecks coexist, and SSC leaves a multi-dimensional collapsed state.  The functional equation is now a multi-variable equation coupling the Laplace transform to multiple boundary terms, and solving the functional equation becomes fundamentally hard.
\section{Conclusion and Future Directions}\label{sec:conclusion}
This tutorial presented the transform method, a framework for
steady-state analysis of Stochastic Processing and Matching Networks.
Starting from the discrete-time single-server queue, we showed how an
exponential test function combined with the complementarity condition
yields an exact functional equation for the MGF (or characteristic
function) of the queue length. \new{Demonstrating its versatility, the three step procedure of the transform method has been extended along two dimensions, viz.,} enriching the single-server model with state-dependent arrivals, customer
abandonment, and Markov modulation
(Section~\ref{sec:single-server-variants}), and moving to
multi-dimensional networks such as load balancing and input-queued
switches (Section~\ref{sec:networks}). We close by highlighting a few
promising directions for future research.
\paragraph{Transient analysis.} Our focus has been steady-state \new{distributions, and we do not study how long it takes to reach the steady-state. Most real world systems do not operate in steady-state because of sudden shocks, such as a batch of unexpected arrivals. 
It is important to get a handle on the transient behavior, especially when it takes a long time to reach steady-state}\new{.}
The exponential test function can in principle be applied to transient
distributions by tracking the time evolution of the MGF rather than
setting its drift to \new{zero. Developing a systematic transient theory based on this approach is an open direction.}

\paragraph{Connections to Stochastic Iterative Algorithms.} \new{While this tutorial focused on queueing models, the ideas presented here are applicable for analysis of any stochastic iterative algorithms, and this is a future research direction. One such example is the work \citet{chen2021stationary}, which considers stochastic approximation (SA) and stochastic gradient descent (SGD) under constant step-size (or learning rate). Unlike the case of diminishing step-sizes, under constant step-size, these algorithms do not converge to a point, and instead have a stationary behavior. It turns out that this steady-state iterates when scaled appropriately in terms of the step-size converge to a limiting distribution, which is a Gaussian in many settings. This result was established in \citet{chen2021stationary} using the Transform method.}

\paragraph{Applications to AI Infrastructure Design} \new{Historically, wired and wireless networks, and data center networks have inspired various stochastic processing network models that are presented in this tutorial. Today, the rapid growth of large language models with novel features such as variable-length token generation, batching, and multistage pipelines is bringing novel challenges to these classical models. This calls for new queueing models to address these problems, and given the power and versatility of the transform method, it is natural to explore its use to study these models. Future work along these lines has a potential to not address practically relevant design and operation questions, but also enrich the transform method by bringing in new ingredients into it.}

\section{Acknowledgment}
The authors acknowledge that large language models, Claude and ChatGPT, were used for editorial refinement, including paraphrasing for clarity and identifying grammatical errors throughout the paper. The final content and interpretations remain the sole responsibility of the authors.
\bibliographystyle{informs2014}
\bibliography{references_q, referencesrl}

\begin{DTonly}
\begin{APPENDIX}{Proof of Lemma~\ref{lem:poisson}}
\label{app:poisson-lemma}
We prove Lemma~\ref{lem:poisson}.  The argument adapts the continuous-time proof of
(\citet{HL-Gro-2026-Markov-Modulated}) to the discrete-time setting.
\begin{proof}{Proof of Lemma~\ref{lem:poisson}}
Recall the Poisson equation for $f$:
\[
  \bar{f} - f(i) \;=\; V_f(i^+) - V_f(i),
  \qquad\text{where}\quad
  V_f(i^+) \;:=\; \sum_{j\in\mathcal{X}} P_{ij}\, V_f(j).
\]
Multiplying both sides by $e^{\theta\epsilon q}$ and taking expectations:
\begin{equation}\label{eq:app-poisson-start}
  \E{e^{\theta\epsilon q}}\,\bar{f}
  - \E{e^{\theta\epsilon q}\,f(X)}
  \;=\;
  \E{e^{\theta\epsilon q}\bigl(V_f(X^+) - V_f(X)\bigr)}.
\end{equation}
It remains to show that the right-hand side equals
$-\theta\epsilon\,\E{e^{\theta\epsilon q}\,V_f(X^+)\,(s(X) - a(X))}$
up to the stated error. We do so by computing the drift of the test function $\phi(q,X) = e^{\theta\epsilon q}\,V_f(X)$.  From the complementarity condition \eqref{eq:key-lemma}, recall that
\[
  e^{\theta\epsilon q(k+1)} \;=\; 1 - e^{-\theta\epsilon u(k)} + e^{\theta\epsilon(q(k) + a_i(X(k),k) - s_i(X(k),k))}.
\]
Therefore,
\begin{align}
  \Delta\phi(q,X)
  &\;=\; e^{\theta\epsilon q(k+1)}\,V_f(X^+) - e^{\theta\epsilon q}\,V_f(X) \notag\\
  &\;=\; V_f(X^+)\bigl(1 - e^{-\theta\epsilon u}\bigr)
         + V_f(X^+)\,e^{\theta\epsilon q}\,e^{\theta\epsilon(a(X) - s(X))}
         - V_f(X)\,e^{\theta\epsilon q} \notag\\
  &\;=\; V_f(X^+)\bigl(1 - e^{-\theta\epsilon u}\bigr)
         + e^{\theta\epsilon q}\Bigl[
           V_f(X^+)\bigl(e^{\theta\epsilon(a(X)-s(X))} - 1\bigr)
           + \bigl(V_f(X^+) - V_f(X)\bigr)
         \Bigr],  \label{eq:app-drift}
\end{align}
where we have omitted the dependence on $k$ for ease of exposition.
Setting $\E{\Delta\phi(q,i)} = 0$ in steady state:
\begin{equation}\label{eq:app-drift-zero}
  \E{e^{\theta\epsilon q}\bigl(V_f(X^+) - V_f(X)\bigr)}
  \;=\;
  -\E{V_f(X^+)\bigl(1 - e^{-\theta\epsilon u}\bigr)}
  \;-\; \E{e^{\theta\epsilon q}\,V_f(X^+)\bigl(e^{\theta\epsilon(a(X)-s(X))} - 1\bigr)}.
\end{equation}
We now bound each term on the right-hand side.
\medskip
\noindent\textbf{Term (i):}
$\E{V_f(X^+)\bigl(1 - e^{-\theta\epsilon u}\bigr)}$.
By H\"older's inequality with $p = 1 + \xi$ and $q = 1 + 1/\xi$:
\[
  \bigl|\E{V_f(X^+)\bigl(1 - e^{-\theta\epsilon u}\bigr)}\bigr|
  \;\le\;
  \E{|V_f(X^+)|^{1+\xi}}^{\frac{1}{1+\xi}}\;
  \E{\bigl|1 - e^{-\theta\epsilon u}\bigr|^{1+\frac{1}{\xi}}}^{\frac{\xi}{1+\xi}}.
\]
The first factor is bounded by assumption.
For the second factor, since $u \le S_{\max}$ and $|1 - e^{-\theta\epsilon u}| \le C|\theta\epsilon u|$ for bounded $u$ and small enough $\epsilon$:
\[
  \E{|1 - e^{-\theta\epsilon u}|^{1+\frac{1}{\xi}}}
  \;\le\;
  C^{1+\frac{1}{\xi}}\,|\theta\epsilon|^{1+\frac{1}{\xi}}\,
  \E{u^{1+\frac{1}{\xi}}}
  \;\le\;
  C'\,\epsilon^{1+\frac{1}{\xi}}\,\E{u}
  \;=\; O\bigl(\epsilon^{2+\frac{1}{\xi}}\bigr),
\]
where we used $\E{u} = \epsilon$ and $u \le S_{\max}$.  Combining:
\[
  \text{Term (i)} \;=\; O\!\Bigl(\epsilon^{\frac{(2+1/\xi)\,\xi}{1+\xi}}\Bigr)
  \;=\; O\!\Bigl(\epsilon^{1+\frac{\xi}{1+\xi}}\Bigr).
\]
\medskip
\noindent\textbf{Term (ii):}
$\E{e^{\theta\epsilon q}\,V_f(X^+)\bigl(e^{\theta\epsilon(a(X)-s(X))} - 1\bigr)}$.
Taylor-expanding $e^{\theta\epsilon(a(X) - s(X))} - 1$:
\[
  e^{\theta\epsilon(a(X) - s(X))} - 1
  \;=\; \theta\epsilon(a(X) - s(X)) + \frac{\theta^2\epsilon^2}{2}(a(X) - s(X))^2 + O(\epsilon^3).
\]
Therefore,
\begin{align*}
  \text{Term (ii)}
  &\;=\; \theta\epsilon\,\E{e^{\theta\epsilon q}\,V_f(X^+)\,(a(X) - s(X))}
         + \frac{\theta^2\epsilon^2}{2}\,
           \E{e^{\theta\epsilon q}\,V_f(X^+)\,(a(X) - s(X))^2}
         + O(\epsilon^3).
\end{align*}
For the second term, since $|e^{\theta\epsilon q}| \le 1$ (because $\theta < 0$ and
$q \ge 0$) and $(a(X) - s(X))^2 \le (A_{\max} + S_{\max})^2$:
\[
  \frac{\theta^2\epsilon^2}{2}\,
  \bigl|\E{e^{\theta\epsilon q}\,V_f(X^+)\,(a(X) - s(X))^2}\bigr|
  \;\le\;
  \frac{\theta^2\epsilon^2}{2}\,(A_{\max}+S_{\max})^2\,\E{|V_f(X^+)|}
  \;=\; O(\epsilon^2).
\]
Hence,
\[
  \text{Term (ii)}
  \; = \; \theta\epsilon\,\E{e^{\theta\epsilon q} \, V_f(X^+) \, (a(X) - s(X))}+ O(\epsilon^2).
\]
\medskip
\noindent\textbf{Combining.}
Substituting Terms~(i) and~(ii) into~\eqref{eq:app-drift-zero}:
\begin{align*}
  \E{e^{\theta\epsilon q}\bigl(V_f(X^+) - V_f(X)\bigr)}
  &\;=\;
  -\theta\epsilon\,\E{e^{\theta\epsilon q}\,V_f(X^+)\,(a(X) - s(X))}
  + O\!\Bigl(\epsilon^{1+\frac{\xi}{1+\xi}}\Bigr).
\end{align*}
Substituting back into~\eqref{eq:app-poisson-start} and rearranging gives
\begin{align*}
  \E{e^{\theta\epsilon q}\,f(X)}
  &=
  \E{e^{\theta\epsilon q}}\,\bar{f}
  - \theta \epsilon \, \E{e^{\theta\epsilon q}\,V_f(X^+) \,(s(X) - a(X))}
  \;+\;
  O\!\Bigl(\epsilon^{1+\frac{\xi}{1+\xi}}\Bigr) \\
  &=
  \E{e^{\theta\epsilon q}}\,\bar{f}
  -
  \theta \epsilon \,\E{e^{\theta\epsilon q}\,V_f(X^+)\,(\mu(X) - \lambda(X))}
  \;+\;
  O\!\Bigl(\epsilon^{1+\frac{\xi}{1+\xi}}\Bigr),
\end{align*}
where the last equality holds by the tower property of expectation. This completes the proof.
\end{proof}
\end{APPENDIX}
\end{DTonly}

%%%%%%%%%%%%%%%%
\end{document}